\documentclass[a4paper,10pt]{article}

\usepackage[english]{babel}
\usepackage[T1]{fontenc}
\usepackage{graphicx}
\usepackage{amsmath, amsfonts,amsthm, mathrsfs, amssymb}
\usepackage{verbatim}

\usepackage{enumerate}
\usepackage[latin1]{inputenc}
\usepackage{geometry}
\newcommand{\B}{{\mathbb B}}
\newcommand{\C}{{\mathbb C}}

\newcommand{\E}{{\mathbb E}}
\newcommand{\F}{{\mathbb F}}

\newcommand{\N}{{\mathbb N}}

\newcommand{\R}{{\mathbb R}}

\newcommand{\T}{{\mathbb T}}
\newcommand{\Z}{{\mathbb Z}}

\newcommand{\cA}{{\mathcal A}}

\newcommand{\cF}{{\mathcal F}}

\newcommand{\cL}{{\mathcal L}}
\newcommand{\cM}{{\mathcal M}}

\newcommand{\cO}{{\mathcal O}}

\newcommand{\cR}{{\mathcal R}}
\newcommand{\cS}{{\mathcal S}}
\newcommand{\cT}{{\mathcal T}}
\newcommand{\cU}{{\mathcal U}}

\newcommand{\cY}{{\mathcal Y}}
\newcommand{\cZ}{{\mathcal Z}}

\renewcommand{\d}{\partial}

\newcommand{\grad}{\nabla}

\newcommand{\la}{\left\langle}
\newcommand{\ra}{\right\rangle}

\newcommand{\jap}[1]{\langle #1 \rangle}

\newcommand{\di}{\rm d}
\newcommand{\sleq}{\lesssim}
\newcommand{\sgeq}{\gtrsim}
\newcommand{\nablac}{\jap{\grad}_c}

\numberwithin{equation}{section}
\newtheorem{theorem}{Theorem}[section]
\newtheorem{lemma}[theorem]{Lemma}
\newtheorem{corollary}[theorem]{Corollary}
\newtheorem{proposition}[theorem]{Proposition}

\newtheorem{remark}[theorem]{Remark}
\title{Dynamics of the nonlinear Klein-Gordon equation in the nonrelativistic limit, II}

\author{
 S. Pasquali 
\footnote{
   Dipartimento di Matematica e Fisica, Universit\`a degli Studi Roma Tre, Largo S. Leonardo Murialdo 1, 00146 Roma. 
\newline
 \textit{Email: } \texttt{spasquali@mat.uniroma3.it}
 }
}

\begin{document}

\maketitle

\begin{abstract}
We study the the nonlinear Klein-Gordon (NLKG) equation on a manifold $M$ 
in the nonrelativistic limit, namely as the speed of light $c$ tends to 
infinity. In particular, we consider an order-$r$ normalized approximation 
of NLKG (which corresponds to the NLS at order $r=1$), and prove that when 
$M=\R^d$, $d \geq 2$,  small radiation solutions of the order-$r$ normalized 
equation approximate solutions of the NLKG up to times of order 
$\cO(c^{2(r-1)})$. 
\end{abstract}

\tableofcontents

\section{Introduction} \label{intro}

This paper is a continuation of \cite{pasquali2017dynamics}. 
In these two papers the nonlinear Klein-Gordon (NLKG) equation 
in the nonrelativistic limit, namely as the speed of light $c$ 
tends to infinity, is studied. 

The nonrelativistic limit for the Klein-Gordon equation 
on $\R^d$ has been extensively studied over more than 30 years, and
essentially all the known results only show convergence of the
solutions of NLKG to the solutions of the approximate equation for
times of order $\cO(1)$. The typical statement ensures convergence locally
uniformly in time. In a first series of results (see
\cite{tsutsumi1984nonrelativistic}, \cite{najman1990nonrelativistic} and 
\cite{machihara2001nonrelativistic}) it was shown that, 
if the initial data are in a certain smoothness class, 
then the solutions converge in a weaker
topology to the solutions of the approximating equation. These are
informally called ``results with loss of smoothness''. Although  
in this paper a longer time convergence is proved, this result also fills in 
this group.

Recently, Lu and Zhang in \cite{lu2016partially} proved a result 
which concerns the NLKG with a quadratic nonlinearity. 
Here the problem is that the typical scale
over which the standard approach allows to control the dynamics is
$\cO(c^{-1})$, while the dynamics of the approximating equation takes 
place over time scales of order $\cO(1)$. In that work the authors are 
able to use a normal form transformation (in a spirit quite different 
from ours) in order to extend the time of validity of the approximation 
over the $\cO(1)$ time scale. 
We did not try to reproduce or extend that result. \\

In \cite{pasquali2017dynamics} Birkhoff normal form methods were used in 
order to extend the approximation up to order $\cO(1)$ to the NLKG equation 
on $M$, $M$ being a compact smooth manifolds or $\R^d$; when $M=\R^d$, 
$d \geq 2$, the approximation of solutions of the \emph{linear} KG equation 
with solutions of the linearized order-$r$ normalized equation up to times 
of order $\cO(c^{2(r-1)})$ is proved.

In this paper we prove a long-time approximation result for the dynamics of 
the NLKG: we consider the NLKG equation on $\R^d$, $d \geq 2$, and we prove 
that for $r>1$ solutions of the order-$r$ normalized equation approximate 
solutions of the NLKG equation up to times of order $\cO(c^{2(r-1)})$.

\vskip 10pt 

The present paper and \cite{pasquali2017dynamics} can be thought as examples 
in which techniques from canonical perturbation theory are used together with 
results from the theory of dispersive equations in order to understand the 
singular limit of some Hamiltonian PDEs. 
In this context, the nonrelativistic limit of the NLKG is a relevant example.

The issue of nonrelativistic limit has been studied 
also in the more general Maxwell-Klein-Gordon system 
(\cite{bechouche2004nonrelativistic}, \cite{masmoudi2003nonrelativistic}), 
in the Klein-Gordon-Zakharov system (\cite{masmoudi2008energy}, 
\cite{masmoudi2010klein}), 
in the Hartree equation (\cite{cho2006semirelativistic}) and 
in the pseudo-relativistic NLS (\cite{choi2016nonrelativistic}). 
However, all these results proved the convergence of the solutions 
\emph{locally uniformly in time}; 
no information could be obtained about the convergence of solutions for 
longer (in the case of NLKG, that means $c$-dependent) timescales. 
On the other hand, in the recent \cite{han2017long}, which studies the 
nonrelativistic limit of the Vlasov-Maxwell system, the authors were able to 
prove a stability result valid for times which are polynomial in terms of the 
speed of light for solutions which lie in a neighbourhood of stable equilibria 
of the system. \\

Another example of singular perturbation problem that has been studied 
with canonical perturbation theory is the problem of the continuous 
approximation of lattice dynamics (see e.g. \cite{bambusi2006metastability}). 
In the framework of lattice dynamics, the approximation has been justified only 
for the typical time scale of averaging theorems, which corresponds to our 
$\cO(1)$ time scale. Hopefully the methods developed in \cite{pasquali2017dynamics} and in the present paper could allow to extend the time of validity of those results. \\

The paper is organized as follows.
In sect. \ref{results} we state the results of the paper, together with 
some examples and comments. 
In sect. \ref{dispKG} we show Strichartz estimates for the linear KG 
equation on $\R^d$. In sect. \ref{Galavmethod} we recall an abstract result 
from \cite{pasquali2017dynamics}; next, in sect. \ref{NLKGappl} we apply the 
abstract theorem to the real NLKG equation, making some explicit computations 
of the normal form at the first and at the second step. 
In sect. \ref{BNFstudy} we study the properties of the normalized equation, 
namely its dispersive properties in the linear case and its well-posedness 
for solutions with small initial data in the nonlinear case. 
In sect. \ref{longtappr} we discuss the approximation for longer 
timescales: in particular, to deduce the latter we will exploit 
some dispersive properties of the KG equation reported in sect. \ref{dispKG}. \\

\emph{Acknowledgments}.
This work is a revised and extended version of a part of the author's PhD thesis. The author would like to thank his supervisor for the PhD thesis Professor Dario Bambusi. 

The author is supported by the ERC grant ``HamPDEs''.

\section{Statement of the Main Results} \label{results}

The NLKG equation describes the motion of a spinless particle with mass $m>0$. 
Consider first the real NLKG
\begin{align} \label{NLKGeq}
\frac{\hbar^2}{2mc^2} u_{tt} - \frac{\hbar^2}{2m} \Delta u +\frac{mc^2}{2} u + \lambda |u|^{2(l-1)}u &= 0,
\end{align}
where $c>0$ is the speed of light, $\hbar>0$ is the Planck constant, 
$\lambda \in \R$, $l \geq 2$, $c>0$.

In the following $m=1$, $\hbar=1$. 
As anticipated above, one is interested in the behaviour of solutions
as $c\to\infty$. 

First it is convenient to reduce equation
\eqref{NLKGeq} to a first order system, by making the following
symplectic change variables 
\begin{align*}
 \psi &:= \frac{1}{\sqrt{2}} \left[ \left(\frac{\nablac}{c} \right)^{1/2} u - i \left(\frac{c}{\nablac}\right)^{1/2}v \right], \; \; v = u_t/c^2,
\end{align*}
where
\begin{equation}
\label{nablac}
\nablac:=(c^2-\Delta)^{1/2}, 
\end{equation}
which reduces \eqref{NLKGeq} to the form 
\begin{align} \label{dsa}
-i \psi_t &= c \nablac \psi 
+ \frac{\lambda}{2^l} \left( \frac{c}{\nablac} \right)^{1/2} 
\left[ \left( \frac{c}{\nablac} \right)^{1/2} (\psi+\bar\psi) \right]^{2l-1},
\end{align}
which is hamiltonian with Hamiltonian function given by
\begin{align} \label{dsa1}
H(\bar\psi,\psi) &= \la \bar{\psi}, c\jap{\grad}_c\psi \ra 
+ \frac{\lambda}{2l} \int \left[ 
\left( \frac{c}{\nablac} \right)^{1/2} \frac{\psi+\bar\psi}{\sqrt{2}} 
\right]^{2l} \di x.
\end{align}

\indent In the following the notation $a \sleq b$ is used to mean: 
there exists a positive constant $K$ that does not depend on $c$ such that $a \leq Kb$.\\

Before discussing the approximation of the solutions of NLKG with 
NLS-type equations, we describe the general strategy we use to get them. 

Remark that Eq. \eqref{NLKGeq} is Hamiltonian with Hamiltonian function
\eqref{dsa1}. If one divides the Hamiltonian by a factor $c^2$ 
(which corresponds to a rescaling of time) 
and expands in powers of $c^{-2}$ it takes the form 
\begin{equation} \label{Hmi}
\langle\psi,\bar \psi\rangle + \frac{1}{c^2} P_c(\psi,\bar \psi)
\end{equation}
with a suitable funtion $P_c$. 
One can notice that this Hamiltonian is a perturbation of 
$h_0:=\langle\psi,\bar \psi\rangle $,
which is the generator of the standard Gauge transform, and 
which in particular admits a flow that is periodic in time. 
Thus the idea is to exploit canonical perturbation theory in
order to conjugate such a Hamiltonian system to a system in normal
form, up to remainders of order $\cO(c^{-2r})$, for any given $r \geq 1$. 

The problem is that the perturbation $P_c$ has a vector field which 
is small only as an operator extracting derivatives: hence, if one Taylor 
expands $P_c$ and its vector field, the number of derivatives extracted at
each order increases. This situation is typical in singular perturbation
problems, and the price to pay to get a normal form is that the remainder 
of the perturbation turns out to be an operator that extracts a large number 
of derivatives. 

In Sect. \ref{NLKGappl} the normal form equation is explicitly computed 
in the case $r=2$, $l=2$:
\begin{align} \label{fdas2}
-i \psi_t \; &= \;  c^2 \psi - \frac{1}{2} \Delta\psi + \frac{3}{4} \lambda |\psi|^2\psi \nonumber \\
&+ \frac{1}{c^2} \left[ \frac{51}{8} \lambda^2 |\psi|^4\psi + \frac{3}{16} \lambda \left(2|\psi|^2 \, \Delta\psi + \psi^2 \Delta\bar\psi + \Delta(|\psi|^2\bar\psi) \right) - \frac{1}{8} \Delta^2\psi \right],
\end{align}
namely a singular perturbation of a Gauge-transformed
NLS equation. If one, after a gauge transformation, 
only considers the first order terms, one has the NLS.

The standard way to exploit such a ``singular'' normal form is to use it
just to construct some approximate solution of the original system, and
then to apply Gronwall Lemma in order to estimate the difference with
a true solution with the same initial datum (see for example 
\cite{bambusi2002nonlinear}).

This strategy works also here, but it only leads to a control of the
solutions over times of order $\cO(c^2)$. When scaled back to the
physical time, this allows to justify the approximation of the solutions 
of NLKG by solutions of the NLS over time scales of order $\cO(1)$, 
{\it on any manifold} admitting a Littlewood-Paley decomposition 
(such as Riemannian smooth compact manifolds, or $\R^d$; 
see the introduction of \cite{bouclet2010littlewood} and section 2.1 of 
\cite{burq2004strichartz} for the construction 
of Littlewood-Paley decomposition on compact manifolds). 

A similar result has been obtained for the case $M=\T^d$ by Faou and Schratz 
\cite{faou2014asymptotic}, who aimed to construct numerical schemes which 
are robust in the nonrelativistic limit.

The idea one uses here in order to improve the time scale of the result
is that of substituting Gronwall Lemma with a more sophisticated tool,
namely dispersive estimates and the retarded Strichartz estimate. 
This can be done each time one can prove a dispersive or a Strichartz 
estimate for the linearization of equation \eqref{dsa} on the approximate 
solution, uniformly in $c$.
Now we state our result for the approximation of small radiation solutions 
of the NLKG equation.

\begin{theorem} \label{NLKGtoNLSradthm}
Consider \eqref{dsa} on $\R^d$, $d \geq 2$. 
Let $r>1$, and fix $k_1 \gg 1$. 
Assume that $l \geq 2$ and $r < \frac{d}{2}(l-1)$. 
Then $\exists$ $k_0=k_0(r)>0$ such that for any $k \geq k_1$ and 
for any $\sigma>0$ the following holds: 
consider the solution $\psi_{r}$ of the normalized equation 
\eqref{simpleq}, with initial datum $\psi_{r,0} \in H^{k+k_0+\sigma+d/2}$. 
Then there exist $\alpha^\ast:=\alpha^\ast(d,l,r)>0$ 
and there exists  $c^\ast:=c^\ast(r,k) > 1$, 
such that for any $\alpha> \alpha^\ast$ and for any $c > c^\ast$, 
if $\psi_{r,0}$ satisfies
\begin{align*}
\|\psi_{r,0}\|_{H^{k+k_0+\sigma+d/2}} &\sleq c^{-\alpha},
\end{align*}
then
\begin{align*}
\sup_{t\in [0,T]} \|\psi(t)-\psi_r(t)\|_{H^k_x} &\sleq \frac{1}{c^2}, \; \; T \sleq c^{2(r-1)},
\end{align*}
where $\psi(t)$ is the solution of \eqref{NLKG} with initial datum $\psi_{r,0}$.
\end{theorem}

\begin{remark}
The assumption of existence of $\psi_r$ up to times of order $\cO(c^{2(r-1)})$ 
is actually a delicate matter. Equation \eqref{fdas2}, for example, is a 
quasilinear perturbation of a fourth-order Schr\"odinger equation (4NLS).
Even if we restrict to the case $r=2$, the issues of global well-posedness 
and scattering for solutions with large initial data for Eq. \eqref{fdas2} 
have not been solved. 
For solutions with small initial data, on the other hand, there are some papers 
dealing with the local well-posedness of 4NLS (see for example \cite{hao2007well}), and with global well-posedness and scattering of 4NLS (see \cite{ruzhansky2016global}).  
In Sec. \ref{higherordWP} we prove the local well-posedness for times of order 
$\cO(c^{2(r-1)})$ for solutions of the order-$r$ normalized equation with small 
initial data under the assumptions that $l \geq 2$ and $r < \frac{d}{2}(l-1)$.
\end{remark}

\begin{remark}
Just to be explicit, we make some examples of Theorem \ref{NLKGtoNLSradthm}. 
For $M=\R^2$ and a nonlinearity of order $2l$, we can justify the approximation of small radiation solutions up to times of order $\cO(c^{2(r-1)})$, for $r < l-1$. For $M=\R^3$ and a nonlinearity of order $2l$, we can justify the approximation of small radiation solutions up to times of order $\cO(c^{2(r-1)})$, for $r < \frac{3}{2}(l-1)$.

On the other hand, when $\frac{d}{2}(l-1) \leq 2$, we cannot justify the 
approximation over long time scales: examples of such cases are the cubic NLKG 
in $2$, $3$ and $4$ dimensions, or the quintic NLKG in $2$ dimensions.
\end{remark}

Before closing the subsection, we remark that the condition on $r$ 
in Theorem \ref{NLKGtoNLSradthm} depends on the assumption under which we were 
able to prove a well-posedness result for the normalized equation, which in 
turn depends on the approach presented recently in \cite{ruzhansky2016global}; 
we do not exclude that this technical condition could be improved.

\section{Dispersive properties of the Klein-Gordon equation} \label{dispKG}

\indent We briefly recall some classical notion of Fourier 
analysis on $\R^d$. Recall the definition of the space of 
Schwartz (or rapidly decreasing) functions,

\begin{align*}
\cS &:= \{ f \in C^\infty(\R^d,\R) | \sup_{x \in \R^d} (1+|x|^2)^{\alpha/2} |\d^\beta f(x)| < + \infty, \; \; \forall \alpha \in \N^d, \forall \beta \in \N^d \}.
\end{align*}

In the following $\la x \ra:=(1+|x|^2)^{1/2}$. \\
Now, for any $f \in \cS$ the \emph{Fourier transform} of $f$, 
$\cF f:\R^d \to \R$, is defined by the following formula

\begin{align*}
\cF f(\xi) &:= (2\pi)^{-d/2} \int_{\R^d} f(x) e^{-i \la x,\xi \ra}\di x, \; \; \forall \xi \in \R^d,
\end{align*}
where $\la \cdot,\cdot \ra$ denotes the scalar product in $\R^d$.

At the beginning we will obtain Strichartz estimates for the linear equation 
\begin{align} \label{KG}
-i \, \psi_t \, &= \, c\jap{\grad}_c \, \psi, \; \; x \in \R^d. 
\end{align}

\begin{proposition} \label{strlin}
Let $d \geq 2$. For any Schr\"odinger admissible couples $(p,q)$ and $(r,s)$, 
namely such that
\begin{align*}
2 \leq p&,r \leq \infty, \\ 
2 \leq q&,s \leq \frac{2d}{d-2}, \\ 
\frac{2}{p}+\frac{d}{q} =\frac{d}{2}, &\; 
\frac{2}{r}+\frac{d}{s} =\frac{d}{2}, \\
(p,q,d),(r,s,d) &\neq (2,+\infty,2),
\end{align*}
one has
\begin{align} \label{strestkg}
\|  \jap{\grad}_c^{\frac{1}{q}-\frac{1}{p}} \; e^{it \; c\jap{\grad}_c} \; \psi_0 \|_{L^p_t L^q_x} \; &\sleq \; c^{\frac{1}{q}-\frac{1}{p}-\frac{1}{2}} \; \| \jap{\grad}_c^{1/2}  \psi_0\|_{L^2}, 
\end{align}
\begin{align} \label{retstrkg}
\left\|  \jap{\grad}_c^{\frac{1}{q}-\frac{1}{p}} \; \int_0^t e^{i(t-s) \; c\jap{\grad}_c} \; F(s) \; \di s \right\|_{L^p_t L^q_x} \; &\sleq \; c^{\frac{1}{q}-\frac{1}{p}+\frac{1}{s}-\frac{1}{r}-1} \; \| \jap{\grad}_c^{\frac{1}{r}-\frac{1}{s}+1} F \|_{L^{r'}_t L^{s'}_x}. 
\end{align}
\end{proposition}

\begin{proof}
By a simple scaling argument, from the following result reported by D'Ancona-Fanelli in \cite{d2008strichartz} for the operator $\jap{\grad}:=\jap{\grad}_1$ 
(for more details see the proof of Proposition 3.1 in \cite{pasquali2017dynamics}).

\begin{lemma} \label{danfanlemma}
For all $(p,q)$ Schr\"odinger-admissible exponents \\
\[ \|e^{i\tau \; \jap{\grad} } \; \phi_0 \|_{L^p_\tau \; W^{\frac{1}{q}-\frac{1}{p}-\frac{1}{2},q }_y} = \; \|\jap{\grad}^{\frac{1}{q}-\frac{1}{p}-\frac{1}{2} } \; e^{it \; \jap{\grad}} \; \phi_0 \|_{L^p_\tau \; L^q_y} \; \leq \; \|\phi_0\|_{L^2_y}. \] \\
\end{lemma}
\end{proof}

\begin{remark}
By choosing $p=+\infty$ and $q=2$, we get the following a priori estimate for 
finite energy solutions of \eqref{KG},
\begin{align*}
\| c^{1/2} \jap{\grad}_c^{1/2} \; e^{it \; c\jap{\grad}_c} \; \psi_0 \|_{L^\infty_t L^2_x} \; &\sleq \; \| c^{1/2} \jap{\grad}_c^{1/2}  \psi_0\|_{L^2}.
\end{align*}
We also point out that, since the operators $\jap{\grad}$ and $\jap{\grad}_c$ 
commute, the above estimates in the spaces $L^p_tL^q_x$ extend to 
estimates in $L^p_tW^{k,q}_x$ for any $k \geq 0$.
\end{remark}

\section{A Birkhoff Normal Form result} \label{Galavmethod} 

\indent Consider the scale of Banach spaces 
$W^{k,p}(M,\C^n \times \C^n) \ni (\psi,\bar\psi)$  
($k \geq 1$, $1<p<+\infty$, $n \in \N_0$) 
endowed by the standard  symplectic form. 
Having fixed $k$ and $p$, and $U_{k,p} \subset W^{k,p}$ open, 
we define the gradient of $H \in C^\infty(U_{k,p},\R)$ w.r.t. $\bar\psi$ 
as the unique function s.t.
\begin{align*}
\la \grad_{\bar\psi} H ,\bar h \ra &= \di_{\bar\psi}H \bar h, \; \; \forall h \in W^{k,p},
\end{align*}
so that the Hamiltonian vector field of a Hamiltonian 
function H is given by \\
\[ X_H(\psi,\bar\psi)=(i\grad_{\bar\psi}H, \; -i\grad_{\psi}H). \]
The open ball of radius $R$ and center $0$ in $W^{k,p}$ will be denoted by 
$B_{k,p}(R)$. \\

\begin{remark} \label{littlepaley}
Let $k \geq 0$, $1 < p < +\infty$, we now introduce the 
Littlewood-Paley decomposition on the Sobolev space $W^{k,p}=W^{k,p}(\R^d)$ 
(see \cite{taylor2011partial}, Ch. 13.5). \\
\indent In order to do this, define the cutoff operators in $W^{k,p}$ 
in the following way: 
start with a smooth, radial nonnegative function $\phi_0: \R^d \to \R$
such that $\phi_0(\xi) = 1$ for $|\xi| \leq 1/2$, and 
$\phi_0(\xi) = 0$ for $|\xi| \geq 1$; 
then define $\phi_1(\xi):=\phi_0(\xi/2)-\phi_0(\xi)$, and set
\begin{align} \label{litpal}
\phi_j(\xi) &:= \phi_1(2^{1-j}\xi), \; \; j \geq 2.
\end{align}
Then $(\phi_j)_{j \geq 0}$ is a partition of unity,
\begin{align*}
\sum_{j \geq 0} \phi_j(\xi) &= 1.
\end{align*}
Now, for each $j \in \N$ and each $f \in W^{k,2}$, we can define $\phi_j(D)f$ by
\begin{align*}
\cF( \phi_j(D)f )(\xi) := \phi_j(\xi)\cF(f)(\xi).
\end{align*}
It is well known that for $p \in (1,+\infty)$ the map $\Phi:L^p(\R^d) \to L^p(\R^d,l^2)$,
\begin{align*}
\Phi(f) &:= (\phi_j(D)f)_{j \in \N},
\end{align*}
maps $L^p(\R^d)$ isomorphically onto a closed subspace of $L^p(\R^d,l^2)$, and we have compatibility of norms (\cite{taylor2011partial}, Ch. 13.5, (5.45)-(5.46)), 
\begin{align*}
K'_p \| f \|_{L^p} \leq \|\Phi(f)\|_{L^p(\R^d,l^2)} &:= \left\| \left[ \sum_{j \in \N} |\phi_j(D)f|^2 \right]^{1/2} \right\|_{L^p} \leq K_p \|f\|_{L^p},
\end{align*}
and similarly for the $W^{k,p}$-norm, i.e. for any $k>0$ and $p \in (1,+\infty)$
\begin{align} \label{compnorms}
K'_{k,p} \| f \|_{W^{k,p}} \leq \left\| \left[ \sum_{j \in \N} 2^{2jk} |\phi_j(D)f|^2 \right]^{1/2} \right\|_{L^p} \leq K_{k,p} \|f\|_{W^{k,p}}.
\end{align}
We then define the cutoff operator $\Pi_N$ by
\begin{align} \label{cutoff}
 \Pi_N\psi := \sum_{j \leq N}\phi_j(D)\psi.
\end{align}
We point out that the Littlewood-Paley decomposition, along with equality 
\eqref{compnorms}, can be extended to compact manifolds (see \cite{burq2004strichartz}), 
as well as to some particular non-compact manifolds (see \cite{bouclet2010littlewood}).
\end{remark}

\indent Now we consider a Hamiltonian system of the form \\
\begin{equation} \label{absH}
H=h_0+ \epsilon \, h + \epsilon \, F, 
\end{equation}
where $\epsilon>0$ is a parameter. We assume that
\begin{itemize}
\item[PER]  $h_0$ generates a linear periodic flow $\Phi^t$ with period $2\pi$, 
\[ \Phi^{t+2\pi} = \Phi^t \; \; \forall t. \]
We also assume that $\Phi^t$ is analytic from $W^{k,p}$ to itself 
for any $k \geq 1$, and for any $p \in (1,+\infty)$;
\item[INV] for any $k\geq 1$, for any $p \in (1,+\infty)$,  
$\Phi^t$ leaves invariant the space $\Pi_jW^{k,p}$ for any $j\geq0$. 
Furthermore, for any $j \geq 0$ 
\[ \pi_j(D) \circ \Phi^t = \Phi^t \circ \pi_j(D); \]
\item[NF] $h$ is in normal form, namely
\[ h \circ \Phi^t = h. \]
\end{itemize}
Next we assume that both the Hamiltonian and the vector field 
of both $h$ and $F$  admit an asymptotic expansion in $\epsilon$ of the form
\begin{align} \label{Hexp}
h \sim  \sum_{j \geq 1} \epsilon^{j-1} h_j, &\; \; F \sim \sum_{j \geq 1} \epsilon^{j-1} F_j, \\
X_h \sim \sum_{j \geq 1} \epsilon^{j-1} X_{h_j}, &\; \; X_F \sim \sum_{j \geq 1} \epsilon^{j-1} X_{F_j},
\end{align}
and that the following properties are satisfied
\begin{itemize}

\item[HVF] There exists $R^\ast>0$ such that for any $j \geq 1$ 
\begin{itemize}
\item[$\cdot$] $X_{h_j}$ is analytic from $B_{k+2j,p}(R^\ast)$ to $W^{k,p}$;
\item[$\cdot$] $X_{F_j}$ is analytic from $B_{k+2(j-1),p}(R^\ast)$ to $W^{k,p}$.
\end{itemize}
Moreover, for any $r \geq 1$ we have that 
\begin{itemize}
\item[$\cdot$] $X_{h-\sum_{j=1}^r \epsilon^{j-1} h_j}$ is analytic from $B_{k+2(r+1),p}(R^\ast)$ to $W^{k,p}$;
\item[$\cdot$] $X_{F - \sum_{j=1}^r \epsilon^{j-1} F_j}$ is analytic from $B_{k+2r,p}(R^\ast)$ to $W^{k,p}$.
\end{itemize}

\end{itemize}

In \cite{pasquali2017dynamics} we proved the following theorem.

\begin{theorem}[see Theorem 4.3 in \cite{pasquali2017dynamics}] \label{normformgavthm}
Fix $r\geq1$, $R>0$, $k_1\gg 1$, $1<p<+\infty$. 
Consider \eqref{absH}, and assume PER, INV 
%(with respect to the Littlewood-Paley decomposition)
, NF and HVF.
Then $\exists$ $k_0=k_0(r)>0$ with the following properties: 
for any $k \geq k_1$ there exists $\epsilon_{r,k,p} \ll 1$ such that 
for any $\epsilon<\epsilon_{r,k,p}$ there exists 
$\cT^{(r)}_\epsilon:B_{k,p}(R) \to B_{k,p}(2R)$ analytic canonical transformation 
such that
\[ H_r := H \circ \cT^{(r)}_\epsilon = h_0 + \sum_{j=1}^r\epsilon^j \cZ_j + \epsilon^{r+1} \; \mathcal{R}^{(r)}, \] \\
where $\cZ_j$ are in normal form, namely
\begin{align} \label{NFthm}
\{\cZ_j,h_0\} &= 0,
\end{align} 
and
\begin{align*} 
\sup_{B_{k+k_0,p}(R)} \|X_{\cZ_{j}}\|_{W^{k,p}} &\leq C_{k,p},
\end{align*} 
\begin{align} \label{Remthm}
\sup_{B_{k+k_0,p}(R)} \|X_{\mathcal{R}^{(r)}}\|_{W^{k,p}} &\leq C_{k,p},
\end{align} 
\begin{align} \label{CTthm}
\sup_{B_{k,p}(R)} \|\cT^{(r)}_\epsilon-id\|_{W^{k,p}} &\leq C_{k,p} \, \epsilon.
\end{align}
In particular, we have that \\
\[ \cZ_1(\psi,\bar\psi) = h_1(\psi,\bar\psi) + \la F_1 \ra(\psi,\bar\psi), \\ \]
where $\la F_1 \ra(\psi,\bar\psi) := \int_0^{2\pi} F_1\circ\Phi^t(\psi,\bar\psi) \frac{\di t}{2\pi}$. \\
\end{theorem}

\subsection{The real nonlinear Klein-Gordon equation} \label{NLKGappl}

We first consider the Hamiltonian of the real non-linear Klein-Gordon equation
with power-type nonlinearity on a smooth manifold $M$ 
($M$ is such the Littlewood-Paley decomposition is well-defined; take, for example, a smooth compact manifold, or $\R^d$). The Hamiltonian is of the form
\begin{align} \label{NLKGham}
 H(u,v) &= \frac{c^2}{2} \la v,v\ra + \frac{1}{2} \la u,\jap{\grad}_c^2u \ra \; + \; \lambda \int \frac{u^{2l}}{2l},
\end{align}
where $\jap{\grad}_c:=(c^2-\Delta)^{1/2}$, $\lambda \in \mathbb{R}$, $l \geq 2$. \\
If we introduce the complex-valued variable 
\begin{align}
 \psi &:= \frac{1}{\sqrt{2}} \left[ \left(\frac{\jap{\grad}_c}{c} \right)^{1/2} u - i \left(\frac{c}{\jap{\grad}_c}\right)^{1/2}v \right], \label{changevar}
\end{align}
(the corresponding symplectic 2-form becomes $i \di\psi \wedge \di\bar\psi$), 
the Hamiltonian \eqref{NLKGham} in the coordinates $(\psi,\bar\psi)$ is
\begin{align} \label{NLKGhamnew}
H(\bar\psi,\psi) &= \la \bar{\psi}, c\jap{\grad}_c\psi \ra 
+ \frac{\lambda}{2l} \int \left[ 
\left( \frac{c}{\jap{\grad}_c} \right)^{1/2} \frac{\psi+\bar\psi}{\sqrt{2}} 
\right]^{2l} \di x.
\end{align}
If we rescale the time by a factor $c^{2}$, the Hamiltonian takes the form 
\eqref{absH}, with $\epsilon = \frac{1}{c^2}$, and 
\begin{align} \label{NLKG}
H(\psi,\bar\psi) &= h_0(\psi,\bar\psi) + \epsilon \, h(\psi,\bar\psi) + \epsilon \, F(\psi,\bar\psi),
\end{align}
where

\begin{align}
h_0(\psi,\bar\psi) &= \la \bar\psi,\psi \ra, \\
h(\psi,\bar\psi) &= \la \bar\psi, \left( c \jap{\grad}_c - c^2 \right)\psi \ra \sim \sum_{j\geq 1}\epsilon^{j-1} \; \la\bar\psi,a_j\Delta^j\psi\ra =: \sum_{j\geq 1}\epsilon^{j-1} h_j(\psi,\bar\psi),  \label{hath} \\
F(\psi,\bar\psi) &= \frac{\lambda}{2^{l+1}l} \int \left[ \left(\frac{c}{\jap{\grad}_c}\right)^{1/2} (\psi+\bar\psi) \right]^{2l} \di x \\
&\sim \frac{\lambda}{2^{l+1}l} \int (\psi+\bar\psi)^{2l} \di x \nonumber \\
&+ \epsilon b_2 \int \left[ (\psi+\bar\psi)^{2l-1}\Delta(\psi+\bar\psi) 
+ \ldots 
+ (\psi+\bar\psi)\Delta((\psi+\bar\psi)^{2l-1}) \right] \di x \nonumber \\
&+ \cO(\epsilon^2) \nonumber \\
&=: \sum_{j \geq 1} \epsilon^{j-1} \, F_j(\psi,\bar\psi), \label{HP}
\end{align}
where $(a_j)_{j \geq 1}$ and $(b_j)_{j \geq 1}$ are real coefficients,
and $F_j(\psi,\bar\psi)$ is a polynomial function of the variables 
$\psi$ and $\bar\psi$ (along with their derivatives) and which admits a
bounded vector field from a neighborhood of the origin in $W^{k+2(j-1),p}$ 
to $W^{k,p}$ for any $1<p<+\infty$. 

This description clearly fits the scheme treated in the previous section,
and one can easily check that assumptions PER, NF and HVF are satisfied. 
Therefore we can apply Theorem \ref{normformgavthm} to the 
Hamiltonian \eqref{NLKG}. \\

\begin{remark} \label{1steprem}
\indent About the normal forms obtained by applying Theorem 
\ref{normformgavthm}, we remark that in the first step (case $r=1$ in 
the statement of the Theorem) the homological equation we get is of the form 
\begin{equation} \label{homeq1step}
\{\chi_1,  h_0 \} + F_1 = \la F_1 \ra, 
\end{equation}
where $F_1(\psi,\bar\psi) = \frac{\lambda}{2^{l+1}l} \int (\psi+\bar\psi)^{2l} \di x$. Hence the transformed Hamiltonian is of the form 
\begin{equation} \label{ham1step}
H_1(\psi,\bar\psi) = h_0(\psi,\bar\psi) + \frac{1}{c^2} \left[ -\frac{1}{2} \la\bar\psi,\Delta\psi\ra + \la F_1 \ra(\psi,\bar\psi) \right] + \frac{1}{c^4} \cR^{(1)}(\psi,\bar\psi),
\end{equation}
where 
\begin{align}
\la F_1 \ra(\psi,\bar\psi) &= \frac{\lambda}{2^{l+1}l} \binom{2l}{l} \int |\psi|^{2l} \; \di x. \label{F1av}
\end{align}
If we neglect the remainder and we derive the corresponding 
equation of motion for the system, we get 
\begin{equation} \label{eqstep1}
 -i \psi_t \; = \psi + \frac{1}{c^2} \left[ -\frac{1}{2} \Delta\psi + \frac{\lambda}{2^{l+1}} \binom{2l}{l} |\psi|^{2(l-1)}\psi \right], \\
\end{equation}
which is the NLS, and the Hamiltonian which generates the canonical 
transformation is given by 
\begin{equation} \label{chi1}
\chi_1(\psi,\bar\psi) = \frac{\lambda}{2^{l+1}l} \sum_{\substack{j=0,\ldots,2l \\ j \neq l}} \frac{1}{i \, 2(l-j)} \binom{2l}{j} \int \psi^{2l-j} \bar\psi^{j} \di x.
\end{equation}
Such computations already appeared in \cite{pasquali2017dynamics}.
\end{remark}

\begin{remark} \label{2steprem}
Now we iterate the construction by passing to the case $r=2$. \\

If we neglect the remainder of order $c^{-6}$, we have that 
\begin{align} 
H \circ \cT^{(1)} &= h_0 + \frac{1}{c^2} h_1 + \frac{1}{c^4} \{\chi_1, h_1\} + \frac{1}{c^4} h_2 + \nonumber \\
&+ \frac{1}{c^2} \la F_1 \ra + \frac{1}{c^4} \{\chi_1, F_1\} + \frac{1}{2c^4} \{\chi_1,\{\chi_1,h_0\}\} + \frac{1}{c^4} F_2 \\
&= h_0 + \frac{1}{c^2} \left[ h_1 + \la F_1\ra \right] + \frac{1}{c^4} \left[ \{\chi_1,h_1\} + h_2 + \{\chi_1,F_1\} + \frac{1}{2} \{ \chi_1, \la F_1 \ra - F_1 \} + F_2 \right],
\end{align}
where $h_1(\psi,\bar\psi) = -\frac{1}{2} \la\bar\psi,\Delta\psi\ra$, and 
$\chi_1$ is of the form \eqref{chi1}. \\

Now we compute the terms of order $\frac{1}{c^4}$. 
\begin{align}
  \{\chi_1, h_1 \} &= \di \chi_1 X_{h_1} =  \frac{\d \chi_1}{\d\psi} \cdot i \frac{\d h_1}{\d\bar\psi} - i \frac{\d \chi_1}{\bar\psi} \frac{\d h_1}{\d\psi} \nonumber \\
&= -\frac{\lambda}{2^{l+3}l} \int \left[ \sum_{\substack{j=0,\ldots,2l-1 \\j \neq l}} \frac{1}{l-j} \binom{2l}{j} (2l-j) \psi^{2l-j-1} \bar\psi^j \right] \, \Delta\psi \; \di x \nonumber \\
&+\frac{\lambda}{2^{l+3}l} \int \left[ \sum_{\substack{j=1,\ldots,2l \\j \neq l}}  \frac{1}{l-j} \binom{2l}{j} j \psi^{2l-j} \bar\psi^{j-1} \right] \, \Delta\bar\psi \; \di x \nonumber \\
&=-\frac{\lambda}{2^{l+3}l} \int \Delta\psi \, \psi^{2l-1} + \Delta\bar{\psi} \, \bar{\psi}^{2l-1} \; \di x \nonumber \\
& -\frac{\lambda}{2^{l+3}l} \int \sum_{\substack{j=1,\ldots,2l-1 \\j \neq l}} \frac{1}{l-j} \binom{2l}{j} \int (2l-j) \psi^{2l-j-1}\bar{\psi}^j \, \Delta\psi - j \psi^{2l-j}\bar{\psi}^{j-1} \, \Delta\bar{\psi} \; \di x, \label{chi1h1}
\end{align}
and since $j \neq l$ in the sum we have that
\begin{align} \label{chi1h1av}
\la \{\chi_1, h_1 \} \ra &= 0.
\end{align}

Next,
\begin{align} \label{h2}
h_2 = -\frac{1}{8} \la \bar\psi,\Delta^2\psi \ra, 
\end{align}

\begin{align}
&\{\chi_1, F_1\} \nonumber \\
 &= \frac{\lambda^2}{2^{2l+3}l^2} \int \left[ \sum_{\substack{ j=0,\ldots,2l-1 \\ j \neq l}} \frac{1}{l-j} \binom{2l}{j} (2l-j) \psi^{2l-j-1}\bar{\psi}^j \right] \left[ \sum_{h=1}^{2l} \binom{2l}{h} h \psi^{2l-h}\bar{\psi}^{h-1} \right] \; \di x  \nonumber \\
&- \frac{\lambda^2}{2^{2l+3}l^2} \int \left[ \sum_{\substack{ j=1,\ldots,2l \\ j \neq l}} \frac{1}{l-j} \binom{2l}{j} j \psi^{l-j}\bar{\psi}^{j-1} \right] \left[\sum_{h=0}^{2l-1} \binom{2l}{h} (2l-h) \psi^{2l-h-1}\bar{\psi}^{h} \right] \; \di x  \nonumber
\end{align}
\begin{align}
&= \frac{\lambda^2}{2^{2l+3}l^2} \sum_{\substack{j,h=1,\ldots,2l-1 \\ j \neq l}} \frac{1}{l-j} \binom{2l}{j} \binom{2l}{h} [(2l-j)h-j(2l-h)] \int \psi^{4l-j-h-1} \bar{\psi}^{j+h-1} \; \di x \nonumber \\
&+ \frac{\lambda^2}{2^{2l+3}l^2} \, 2 \int \psi^{2l-1} \left[ \sum_{h=1}^{2l} \binom{2l}{h} h \psi^{2l-h} \bar{\psi}^{h-1} \right] \; \di x \nonumber \\
&+ \frac{\lambda^2}{2^{2l+3}l^2} \, 2l \int \left[ \sum_{\substack{j=0,\ldots,2l-1 \\ j \neq l}} \frac{1}{l-j} \binom{2l}{j} (2l-j) \psi^{2l-j-1} \bar{\psi}^j \right] \bar{\psi}^{2l-1} \; \di x \nonumber \\
&+ \frac{\lambda^2}{2^{2l+3}l^2} \, 2 \int \bar{\psi}^{2l-1} \left[ \sum_{h=0}^{2l-1} \binom{2l}{h} (2l-h) \psi^{2l-h-1} \bar{\psi}^h \right] \; \di x \nonumber \\
&- \frac{\lambda^2}{2^{2l+3}l^2} \, 2l \int \left[ \sum_{\substack{j=1,\ldots,2l \\ j \neq l}} \frac{1}{l-j} \binom{2l}{j} j \psi^{2l-j}\bar{\psi}^{j-1} \right] \psi^{2l-1} \; \di x, \label{chi1F1}
\end{align}

\begin{align}
\la \{\chi_1, F_1\} \ra &= \lambda^2 K(l) \int |\psi|^{2(2l-1)} \; \di x, \label{chi1F1av} \\
K(l) &:= \frac{1}{2^{2l+3}l^2} \left\{ \left( \sum_{\substack{ j,h=1,\ldots,2l-1 \\ j \neq l \\ j+h=2l }} \frac{1}{l-j} \binom{2l}{j} \binom{2l}{h} [(2l-j)h-j(2l-h)] \right) + 16 l \right\},
\end{align}
where $K(l)>0$ by the conditions on $j$ and $h$ in the sum.

Then,
\begin{align}
&\{ \chi_1, \la F_1 \ra \} \nonumber \\
&= \frac{\lambda^2}{2^{2l+3}l^2} \binom{2l}{l} \int \sum_{\substack{j=0,\ldots,2l-1 \\ j \neq l}} \frac{1}{l-j} \binom{2l}{j} (2l-j)l \, \psi^{2l-j-1} \bar{\psi}^j \psi^l \bar{\psi}^{l-1} \; \di x \nonumber \\
&- \frac{\lambda^2}{2^{2l+3}l^2} \binom{2l}{l} \int \sum_{\substack{j=1,\ldots,2l \\ j \neq l}} \frac{1}{l-j} \binom{2l}{j} jl \, \psi^{2l-j}\bar{\psi}^{j-1} \psi^{l-1} \bar{\psi}^l \; \di x \nonumber \\
&= \frac{\lambda^2}{2^{2l+3}l^2} \binom{2l}{l} \left[ \binom{2l}{l} \, 2 \int \psi^{3l-1} \bar{\psi}^{l-1} +\psi^{l-1} \bar{\psi}^{3l-1} \; \di x + \sum_{\substack{j=1,\ldots,2l-1 \\j \neq l}}2l \binom{2l}{j} \int \psi^{3l-j-1} \bar{\psi}^{j+l-1} \; \di x \right], \label{chi1avF1}
\end{align}
and since $j \neq l$ in the sum we have that
\begin{align}
\la \{ \chi_1, \la F_1 \ra \} \ra &= 0.
\end{align}

Furthermore,
\begin{align}
F_2 &= \frac{\lambda}{2^{l+3}l} \, 2l \int (\psi+\bar\psi)^{2l-1} \, \Delta(\psi+\bar\psi) \; \di x \nonumber \\
&= \frac{\lambda}{2^{l+2}} \sum_{j=0}^{2l-1} \binom{2l-1}{j} \int \psi^{2l-j-1}\bar{\psi}^j (\Delta \psi + \Delta \bar{\psi}) \; \di x, \label{F2}
\end{align}

\begin{align}
\la F_2 \ra &= \frac{\lambda}{2^{l+2}} \int \binom{2l-1}{l} \psi^{l-1}\bar{\psi}^l \Delta\psi + \binom{2l-1}{l-1} \psi^l \bar{\psi}^{l-1} \Delta\bar{\psi} \; \di x \nonumber \\
&= \frac{\lambda}{2^{l+2}} \binom{2l-1}{l} \int |\psi|^{2(l-1)}( \bar{\psi} \Delta\psi + \psi \Delta\bar{\psi} ) \; \di x
\end{align}

Hence, up to a remainder of order $O\left(\frac{1}{c^6}\right)$, we have that
\begin{align} \label{ham2step}
H_2 &= h_0 + \frac{1}{c^2} \int \left[ -\frac{1}{2}  \la \bar\psi,\Delta\psi \ra +  \frac{\lambda}{2^{l+1}l} \binom{2l}{l} |\psi|^{2l} \right] \; \di x  \nonumber\\
&+  \frac{1}{c^4} \int \left[ \lambda^2 K(l) |\psi|^{2(2l-1)} + \frac{\lambda}{2^{l+2}} \binom{2l-1}{l} |\psi|^{2(l-1)}( \bar{\psi} \Delta\psi + \psi \Delta\bar{\psi} ) - \frac{1}{8} \la \bar\psi, \Delta^2\psi \ra \right] \; \di x,
\end{align}
which, by neglecting $h_0$ (that yields only a gauge factor) and 
by rescaling the time, leads to the following equations of motion
\begin{align} \label{eqstep2}
-i \psi_t  &= - \frac{1}{2} \Delta\psi + \frac{\lambda}{2^{l+1}} \binom{2l}{l} |\psi|^{2(l-1)}\psi + \frac{1}{c^2} \left[ - \frac{1}{8} \Delta^2\psi + \lambda^2 K(l) \, (2l-1) |\psi|^{4(l-1)}\psi \right] \nonumber \\
&+ \frac{1}{c^2} \left[ \frac{\lambda}{2^{l+2}} \binom{2l-1}{l} \left(l |\psi|^{2(l-1)} \, \Delta\psi + (l-1) |\psi|^{2(l-2)} \psi^2 \Delta\bar\psi + \Delta(|\psi|^{2(l-1)}\bar\psi) \right) \right],
\end{align}
which for example in the case of a cubic nonlinearity ($l=2$) reads
\begin{align} \label{eqstep2l2}
-i \psi_t \; &= \;  - \frac{1}{2} \Delta\psi + \frac{3}{4} \lambda |\psi|^2\psi \nonumber \\
&+ \frac{1}{c^2} \left[ \frac{51}{8} \lambda^2 |\psi|^4\psi + \frac{3}{16} \lambda \left(2|\psi|^2 \, \Delta\psi + \psi^2 \Delta\bar\psi + \Delta(|\psi|^2\bar\psi) \right) - \frac{1}{8} \Delta^2\psi \right].
\end{align}

Eq. \eqref{eqstep2l2} is the nonlinear analogue of a linear higher-order 
Schr\"odinger equation that appears in \cite{carles2012higher} and 
\cite{carles2015higher} in the context of semi-relativistic equations. 
\end{remark}

\section{Properties of the normal form equation} \label{BNFstudy}

\subsection{Linear case} \label{BNFlinsec}

\indent Now let $r \geq 1$, $d \geq 2$. 
In \cite{carles2012higher} and \cite{carles2015higher} the authors proved 
that the linearized normal form system, namely the one that corresponds 
(up to a rescaling of time by a factor $c^2$) to 
\begin{align} 
-i \dot{\psi_r} &= X_{h_0 + \sum_{j=1}^r \epsilon^j h_j}(\psi_r), \label{schrordr} \\
\psi_r(0) &= \psi_0, \nonumber
\end{align}
admits a unique solution in $L^\infty(\R)H^{k+k_0}(\R^d)$  
(this is a simple application of the properties of the Fourier transform), and 
by a perturbative argument they also proved the global existence also for the 
higher oder Schr\"odinger equation with a bounded time-independent potential.

Moreover, by following the arguments of Theorem 4.1 in \cite{kim2012global} 
and Lemma 4.3 in \cite{carles2015higher} one obtains the following 
dispersive estimates and local-in-time Strichartz estimates 
for solutions of the linearized normal form equation \eqref{schrordr}.

\begin{proposition}
Let $r \geq 1$ and $d \geq 2$, and denote by $\cU_r(t)$ the evolution 
operator of \eqref{schrordr} at the time $c^2t$ ($c \geq 1$, $t>0$). 
Then one has the following local-in-time dispersive estimate 
\begin{align} \label{locdispest}
\| \cU_r(t) \|_{L^1(\R^d) \to L^\infty(\R^d)} &\sleq 
c^{d \left( 1-\frac{1}{r} \right)} |t|^{-d/(2r)}, \; \; 0<|t| \sleq c^{2(r-1)}.
\end{align}
On the other hand, $\cU_r(t)$ is unitary on $L^2(\R^d)$. \\
Now introduce the following set of admissible exponent pairs:
\begin{align} \label{Deltar}
\Delta_r &:= \left\{ (p,q): (1/p,1/q) \; \text{lies in the closed quadrilateral ABCD}\right\},
\end{align}
where 
\[ A=\left(\frac{1}{2},\frac{1}{2}\right), \; \; 
B=\left(1,\frac{1}{\tau_r}\right), \; \; C=(1,0), \; \;
D=\left(\frac{1}{\tau_r'},0\right), \; \; \tau_r = \frac{2r-1}{r-1}, \; \; 
\frac{1}{\tau_r} + \frac{1}{\tau_r'} = 1. \] 
Then for any $(p,q)\in\Delta_r \setminus \{(2,2),(1,\tau_r),(\tau_r',\infty)\}$ 
\begin{align}
\| \cU_r(t) \|_{L^p(\R^d) \to L^q(\R^d)} &\sleq 
c^{d \left( 1-\frac{1}{r} \right) \left( \frac{1}{p}-\frac{1}{q} \right)} |t|^{-\frac{d}{2r} \left( \frac{1}{q}-\frac{1}{p} \right) }, \; \; 0<|t| \sleq c^{2(r-1)}, \label{LpLqhighschr}
\end{align}
\end{proposition}

Let $r \geq 1$ and $d \geq 2$: in the following lemma $(p,q)$ is called an 
order-$r$ admissible pair when $2 \leq p,q \leq +\infty$ for 
$r \geq 2$ ($2 \leq q \leq 2d/(d-2)$ for $r=1$), and
\begin{align} \label{admhighschr}
\frac{2}{p} + \frac{d}{rq} &= \frac{d}{2r}.
\end{align}

\begin{proposition}
Let $r \geq 1$  and $d \geq 2$, and denote by $\cU_r(t)$ the evolution operator 
of \eqref{schrordr} at the time $c^2t$ ($c \geq 1$, $t>0$). 
Let $(p,q)$ and $(a,b)$ be order-$r$ admissible pairs, then 
for any $T \sleq c^{2(r-1)}$
\begin{align}
\| \cU_r(t)\phi_0 \|_{L^p([0,T])L^q(\R^d)} &\sleq 
c^{ d \left( 1-\frac{1}{r} \right) \left(\frac{1}{2} -\frac{1}{q} \right) } \|\phi_0\|_{L^2(\R^d)} 
= c^{ \left(1-\frac{1}{r}\right) \frac{2r}{p} } \|\phi_0\|_{L^2(\R^d)}, \label{strhighschr} \\
\left\| \int_0^t \cU_r(t-\tau)\phi(\tau) \di\tau \right\|_{L^p([0,T])L^q(\R^d)} &\sleq c^{\left(1-\frac{1}{r}\right)2r\left(\frac{1}{p}+\frac{1}{a}\right)} \|\phi\|_{L^{a'}([0,T])L^{b'}(\R^d)}. \label{retstrhighschr}
\end{align}
\end{proposition}

\subsection{Well-posedness of higher order nonlinear Schr\"odinger equations with small data} \label{higherordWP}

Here we discuss the local well-posedness of 
\begin{align} \label{highordschr}
-i \psi_t &= A_{c,r} \psi + P( (\d^\alpha_x\psi)_{|\alpha| \leq 2(r-1)},(\d^\alpha_x\bar\psi)_{|\alpha| \leq 2(r-1)} ), \; \; t \in I, \; x \in \R^d,\\
\psi(0,x) &= \psi_0(x),
\end{align}
where $r \geq 2$, $I:=[0,T]$, $T>0$,
\begin{align*}
A_{c,r} &= c^2-\sum_{j=1}^r \frac{\Delta^j}{c^{2(j-1)}}, \; \;  c \geq 1,
\end{align*}
and $P$ is an analytic function at the origin of the form
\begin{align}
P(z) &= \sum_{m+1 \leq |\beta| < M} a_\beta z^\beta, \; \; |a_\beta| \leq K^{|\beta|}, \; |z| \ll 1,
\end{align}
where $M > m \geq 2$, $m,M \in \N$.

We will exploit this result during the proof of Theorem \ref{NLKGtoNLSradthm}. We will adapt an argument of \cite{ruzhansky2016global} in order to show the local well-posedness of Eq. for data with small norm in the so-called modulation spaces.

Modulation spaces $M^s_{p,q}$ ($s \in \R$, $0 < p,q < +\infty$) were introduced by Feichtinger, and they can be seen as a variant of Besov spaces, in the sense that they allow to perform a frequency decomposition of operators, and to study their properties with respect to lower and higher frequencies. This spaces were recently used in order to prove global well-posedness and scattering for small data for nonlinear dispersive PDEs, especially in the case of derivative nonlinearities (see for example \cite{wang2007global}, \cite{wang2009global} and \cite{ruzhansky2016global}).
We refer to \cite{ruzhansky2012modulation} for a survey about modulation spaces and nonlinear evolution equations.

We define the norm on modulation spaces via the following decomposition: let $\sigma:\R^d \to \R$ be a function such that 
\begin{align*}
supp(\sigma) &\subset [-3/4,3/4]^d,
\end{align*}
and consider a function sequence $(\sigma_k)_{k \in \Z^d}$ satysfying
\begin{align}
\sigma_k(\cdot) &= \sigma(\cdot-k), \label{decpr1} \\
\sum_{k \in \Z^d} \sigma_k(\xi) &= 1, \; \; \forall \xi \in \R^d. \label{decpr2}
\end{align}
Denote by
\begin{align*}
\cY_d := \{ (\sigma_k)_{k \in \Z^d}: (\sigma_k)_{k \in \Z^d} \text{satisfies} \eqref{decpr1}-\eqref{decpr2} \}.
\end{align*} 
Let $(\sigma_k)_{k \in \Z^d} \in \cY_d$, and define the frequency-uniform decomposition operators
\begin{align} \label{frdecop}
\square_k := \cF^{-1}\sigma_k\cF,
\end{align}
where by $\cF$ we denote the Fourier transform on $\R^d$, then we define the 
modulation spaces $M^s_{p,q}(\R^d)$ via the following norm,
\begin{align} \label{modspacenorm}
\|f\|_{M^s_{p,q}(\R^d)} := \left( \sum_{k \in \Z^d} \la k \ra^{sq} \|\square_k f\|_p^q \right)^{1/q}, \; \; s \in \R, 0< p,q < +\infty.
\end{align}
Actually, in our application we will always be interested in the spaces $M^s_{p,1}(\R^d)$ with $s \in \R$ and $p>1$. We just mention some properties of modulation spaces.

\begin{proposition} \label{modspaceprop}
Let $s,s_1,s_2 \in \R$ and $1 < p,p_1,p_2 < +\infty$.
\begin{enumerate}
\item $M^s_{p,1}(\R^d)$ is a Banach space; \\
\item $\cS(\R^d) \subset M^s_{p,1}(\R^d) \subset \cS'(\R^d)$; \\
\item $\cS(\R^d)$ is dense in $M^s_{p,1}(\R^d)$; \\
\item if $s_2 \leq s_1$ and $p_1 \leq p_2$, then $M^{s_1}_{p_1,1} \subseteq M^{s_2}_{p_2,1}$; \\
\item $M^0_{p,1}(\R^d) \subseteq L^\infty(\R^d) \cap L^p(\R^d)$;
\item let $\tau(p) = max\left( 0, d(1-1/p), d/p \right)$ and $s_1 > s_2 + \tau(p)$, then $W^{s_1,p}(\R^d) \subset M^{s_2}_{p,1}(\R^d)$; \\
\item let $s_1 \geq s_2$, then $M^{s_1}_{p,1}(\R^d) \subset W^{s_2,p}(\R^d)$.
\end{enumerate}
\end{proposition} 

The last two properties are not trivial, and have been proved in \cite{kobayashi2011inclusion}.

We also introduce other spaces which are often used in this context: the anisotropic Lebesgue space $L^{p_1,p_2}_{x_i;(x_j)_{j\neq i},t}$,
\begin{align*}
\|f\|_{ L^{p_1,p_2}_{x_i;(x_j)_{j\neq i},t} } &:= \left\| \|f\|_{L^{p_2}_{x_1,\ldots,x_{i-1},x_{i+1},\ldots,x_d,t}(\R^{d-1} \times I)} \right\|_{L^{p_1}_{x_i}(\R)},
\end{align*}
and, for any Banach space $X$, the spaces $l^{1,s}_\square(X)$ and $l^{1,s}_{\square,i}(X)$,
\begin{align}
\|f\|_{l^{1,s}_\square(X)} &:= \sum_{k \in \Z^d} \la k \ra^s \|\square_k f\|_X,  \label{norm1} \\
\|f\|_{l^{1,s}_{\square,i,c}(X)} &:= \sum_{k \in \Z^d_i} \la k \ra^s \|\square_k f\|_X, \; \; \Z^d_i:= \{ k \in \Z^d: |k_i|=\max_{1 \leq j \leq d}|k_j|, |k_i|>c\}. \label{norm2}
\end{align}
For simplicity, we write $l^1_\square(X)=l^{1,0}_\square(X)$ and $M^s_{p,1}=M^s_{p,1}(\R^d)$.

\begin{proposition} \label{LWPhighordschr}
Let $d \geq 2$, $m\geq 2$, $m > 4r/d$ and $s > 2(r-1)+1/m$.
\begin{enumerate}
\item[(i)] There exist $c_0>1$ and $\delta_0=\delta_0(d,m,r)>0$ such that for any $c \geq c_0$, for any $\delta>\delta_0$ and for any $\psi_0 \in M^s_{2,1}$ with $\|\psi_0\|_{M^s_{2,1}} \leq c^{-\delta}$ the equation \eqref{highordschr} admits a unique solution $\psi \in C(I,M^s_{2,1}) \cap D$, where $T=T( \|\psi_0\|_{M^s_{2,1}} ) = \cO( c^{2(r-1)} )$, and
\begin{align} \label{normprop}
\|\psi\|_D &= \sum_{\alpha=0}^{2(r-1)} \sum_{i,l=1}^d \|\d_{x_l}^\alpha\psi\|_{ l^{1,s-r+1/2}_{\square,i,c}(L^{\infty,2}_{x_i;(x_j)_{j\neq i},t}) \cap l^{1,s}_\square(L^{m,\infty}_{x_i;(x_j)_{j\neq i},t}) \cap l^{1,s+1/m}_\square(L^\infty_t L^2_x \cap L^{2+m}_{t,x}) } \; \sleq \; c^{-\delta}.
\end{align}
\item[(ii)] Moreover, if $s \geq s_0(d):=d+2+\frac{1}{2}$, then there exists $\delta_1=\delta_1(d,m,r)>0$ such that for any $c \geq c_0$, for any $\delta>\delta_1$ and for any $\psi_0 \in M^s_{2,1}$ with $\|\psi_0\|_{M^s_{2,1}} \leq c^{-\delta}$  the equation \eqref{highordschr} admits a unique solution $\psi \in C(I,H^s)$, where $T=T( \|\psi_0\|_{M^s_{2,1}} ) = \cO( c^{2(r-1)} )$, and
\begin{align} \label{normsolsobest}
\|\psi(t)\|_{H^{s}} &\sleq c^{-\delta}, \; \; |t| \sleq c^{2(r-1)}.
\end{align}
\end{enumerate}
\end{proposition}

From the above Proposition and from the embedding 
$H^{s+\sigma+d/2} \subset M^s_{2,1}$ for any $\sigma>0$ 
we can deduce

\begin{corollary} \label{LWPBNFr}
Let $d \geq 2$, $l\geq 2$, $r<\frac{d}{2}(l-1)$ and $s>2(r-1)+\frac{1}{2(l-1)}$.
Then there exist $c_0>1$,  $\delta_0=\delta_0(d,l,r)>0$ and $\delta_1=\delta_1(d,l,r)>0$ such that for any $c \geq c_0$, for any $\delta>\max(\delta_0,\delta_1)$, for any $\sigma>0$ and for any $\psi_0 \in H^{s+\sigma+d/2}$ with $\|\psi_0\|_{H^{s+\sigma+d/2}} \leq c^{-\delta}$ the normal form equation for \eqref{NLKGham} admits a unique solution $\psi \in C([0,T],H^{s+\sigma+d/2}) \cap D$, where $T=T( \|\psi_0\|_{H^{s+\sigma+d/2}} ) = \cO( c^{2(r-1)} )$, and \eqref{normprop} holds.
Furthermore, we have that $\psi \in L^\infty(I)H^{s+\sigma+d/2}(\R^d)$, and 
\begin{align} \label{normsolsobest2}
\|\psi(t)\|_{H^{s+\sigma+d/2}} &\sleq c^{-\delta}, \; \; |t| \sleq c^{2(r-1)}.
\end{align}
\end{corollary}

Since the nonlinearity in Eq. \eqref{highordschr} involves derivatives, this could cause a loss of derivatives as long as we rely only on energy estimates, on dispersive estimates or on Strichartz estimates. In order to overcome such a problem, we will study the time decay of the operator $\cU_r(t):=e^{itA_{c,r}}$, its local smoothing property, Strichartz estimates with $\square_k$-decomposition and maximal function estimates in the framework of frequency-uniform localization.

The rest of this subsection is devoted to the proof of Proposition 
\ref{LWPhighordschr}. For convenience, we will always use the following 
function sequence $(\sigma_k)_{k \in \Z^d}$ to define modulation spaces.

\begin{lemma}
Let $(\eta_k)_{k \in \Z} \in \cY_1$, and assume that $\text{supp}(\eta_k) \subset [k-2/3,k+2/3]$. Consider
\begin{align} \label{decompfun}
\sigma_k(\xi) := \eta_{k_1}(\xi_1) \cdots \eta_{k_d}(\xi_d), \; k = (k_1,\ldots,k_d) \in \Z^d,
\end{align}
then $(\sigma_k)_{k \in \Z^d} \in \cY_d$.
\end{lemma}

For convenience, we also write
\begin{align} \label{decop}
\tilde\sigma_k = \sum_{\|l\|_\infty \leq 1} \sigma_{k+l}, \; &\; \tilde\square_k = \sum_{\|l\|_\infty \leq 1} \square_{k+l}, \; k \in \Z^d,
\end{align}
and one can check that
\begin{align} \label{decoprop}
\tilde\sigma_k \sigma_k = \sigma_k, \; &\; \tilde\square_k \circ \square_k = \square_k, \; k \in \Z^d.
\end{align}

We also write $\cA_rf(t,x) := \int_0^t \cU_r(t-\tau)f(\tau,x)\di\tau$.

\subsubsection{Time decay} \label{tdecsubsec}

Now, the time-decay of the operator $\cU_r(t)$ is known (see \eqref{locdispest}), but now we are interested in its frequency-localized version, and we want to consider lower, medium and higher frequency separately. For simplicity we discuss the case $r=2$, and we defer to the the end of this section a remark about the case $r>2$. So, consider
\begin{align*}
\cU_2(t) &= e^{itA_{c,2}} = e^{ic^2t} \; \cF^{-1} e^{it(|\xi|^2-\frac{|\xi|^4}{c^2})} \cF,
\end{align*}
and write $\epsilon=c^{-2}$. It is known that the time decay of $\cU_2(t)$ is determined by the critical points of $P_2(|\xi|)= |\xi|^2-\epsilon |\xi|^4$. Notice that $P'_2(R)=4R(\epsilon^{1/2}R+\frac{1}{\sqrt{2}})(\epsilon^{1/2}R-\frac{1}{\sqrt{2}})$, the singular points of $P_2$ are $\xi=0$ and the points of the sphere $\xi=(2\epsilon)^{-1/2}$. To handle these points, we exploit Littlewood-Paley decomposition, Van der Corput lemma and some properties of the Fourier transform of radial functions.

Indeed, it is known that the Fourier transform of a radial function $f$ is radial, 
\begin{align*}
\cF f(\xi) &= 2\pi \int_0^\infty f(R) R^{d-1}(R|\xi|)^{-(d-2)/2} J_{\frac{d-2}{2}}(R|\xi|) dR,
\end{align*}
where $J_m$ is the order $m$ Bessel function,
\begin{align*}
J_m(R) &= \frac{(R/2)^m}{\Gamma(m+1/2)\pi^{1/2}} \int_{-1}^1 e^{iRt} (1-t^2)^{m-1/2} dt, \; \; m>-1/2.
\end{align*}
By following the computations in \cite{ruzhansky2016global} we obtain that
\begin{align}
\cF f(s) &= K_d \pi \int_0^\infty f(R) R^{d-1} e^{-iRs} \bar h(Rs) \di R + K_d \pi \int_0^\infty f(R) R^{d-1} e^{iRs} h(Rs) \di R, \; \; K_d>0, \label{FTradial} \\
|h^{(k)}(R)| &\leq K_d (1+R)^{- \frac{d-1}{2}-k}, \; \; \forall k \geq 0. \label{FTradial2}
\end{align}

Now we make a Littlewood-Paley decomposition of the frequencies: choose $\rho$ a smooth cut-off function equal to $1$ in the unit ball and equal to $0$ outside the ball of radius $2$, write $\phi_0=\rho(\cdot)-\rho(2\cdot)$, $\phi_j(\cdot)= \cF^{-1}\phi_0(2^{-j}\cdot)\cF$, $j \in \Z$, and consider
\begin{align}
\cU_2(t)\psi_0 &= \sum_{|j| \leq K} \phi_j(D)\cU_2(t)\psi_0 + \sum_{j<- K} \phi_j(D)\cU_2(t)\psi_0 + \sum_{j>K} \phi_j(D)\cU_2(t)\psi_0 \nonumber \\
&=: P_= \, \cU_2(t)\psi_0 + P_<\,\cU_2(t)\psi_0 + P_>\,\cU_2(t)\psi_0, \label{lpdec}
\end{align}
where 
\begin{align} \label{threshold}
K &:= K(\epsilon) \; = \; 10-\frac{1}{2} \lceil \log_2 \epsilon\rceil. 
\end{align}

Notice that the singular point $R=0$ is in the support set of $\cF(P_= \, \cU_2 (t)\psi_0)$. Roughly speaking, if $j <-K$, the dominant term in $P_2(R)$ is $R^2$, while if $j>K$ the dominant term in $P_2(R)$ is $\epsilon R^4$; hence, by \eqref{locdispest}
\begin{align}
\| P_< \, \cU_2(t)\psi_0 \|_{L^\infty} &\sleq |t|^{-d/2} \|\psi_0\|_{L^1}, \label{dispestlow} \\
\| P_> \, \cU_2(t)\psi_0 \|_{L^\infty} &\sleq c^{d/2}|t|^{-d/4} \|\psi_0\|_{L^1}, \; 0 < |t| \sleq c^2. \label{dispesthigh}
\end{align}

The time decay estimate for $P_= \, \cU_2(t)\psi_0$ is more difficult, since $P_2(R)$ has a singular point in $R=R_1:=(2\epsilon)^{-1/2}$, which corresponds to the sphere $|\xi|=R_1$ in the support set of $\cF(P_= \, \cU_2 (t)\psi_0)$. We notice that also the point that satisfies $P_2''(R)=0$, $R=(6\epsilon)^{-1/2}$, corresponds to a sphere $\xi=R_2$ contained in the support set of $\cF(P_= \, \cU_2 (t)\psi_0)$; we shall use this fact later. \\

In order to handle the singular point $R_1$, we perform another decomposition around the sphere $|\xi|=R_1$. Denote $\tilde\rho(\cdot)=\rho(2^{-K}\cdot)-\rho(2^{(K+1)}\cdot)$, then $P_= = \cF^{-1} \tilde\rho \cF$; write $P_k = \cF^{-1} \phi_k(|\xi|-R_1)\cF$, we get
\begin{align} \label{sing2dec}
\sum_{|j| \leq K} \phi_j(D)\cU_2(t)\psi_0 &= \sum_{k \in \Z} P_= P_k \, \cU_2(t)\psi_0
\end{align}

By Young's inequality
\begin{align}
\|P_= P_k \, \cU_2(t)\psi_0\|_{L^\infty} &\sleq \| \cF^{-1}\left( \tilde\rho \phi_k(|\xi|-R_1) e^{-itP_2(|\xi|)}  \right) \|_{L^\infty} \|\psi_0\|_{L^1}.
\end{align}
Moreover,
\begin{align*}
&\cF^{-1}\left( \tilde\rho \phi_k(|\xi|-R_1) e^{-itP_2(|\xi|)}  \right) \\
&\stackrel{\eqref{FTradial}}{=} K_d \pi \int_0^\infty  R^{d-1} \tilde\rho(R)\phi_k(R-R_1) e^{-itP_2(R)-iR|x|} \bar h(R|x|) \di R \\
&\; \; + K_d \pi \int_0^\infty  R^{d-1} \tilde\rho(R)\phi_k(R-R_1) e^{-itP_2(R)+iR|x|}  h(R|x|) \di R \\
&=:A_k(|x|)+B_k(|x|).
\end{align*}
In order to estimate $A_k(s)$ we rewrite it as
\begin{align}
A_k(s) &= K_d\pi \left( \int_{R_1}^\infty+\int_0^{R_1} \right) R^{d-1} \tilde\rho(R)\phi_k(R-R_1) e^{-itP_2(R)-iRs} \bar h(Rs) \di R \\
&=: A_k^{(1)}(s)+A_k^{(2)}(s).
\end{align}

We begin by estimating $A_k^{(1)}$: notice that $A_k^{(1)}(s)$ for $k>K+2$, hence we can assume that $k \leq K+2$. By a change of variables we obtain
\begin{align*}
A_k^{(1)}(s) &\stackrel{R=R_1+2^k\sigma}{=} 2^k K_d\pi e^{-iR_1s} \int_{1/2}^2 F(\sigma) e^{it 2^{2k} \tilde{P_2}(\sigma) } \di\sigma, \\
F(\sigma) &:= (R_1+2^k \sigma)^{d-1} \tilde\rho(R_1+2^k\sigma)\phi_0(\sigma) \bar h((R_1+2^k\sigma)s), \\
\tilde{P_2}(\sigma) &:= (2^{2k}t)^{-1} (t P_2(R_1+2^k\sigma)-2^k\sigma s).
\end{align*}
One can check that
\begin{align*}
|\tilde{P_2}'(\sigma)| &= \left| 4(R_1+2^k\sigma)(2R_1+2^k\sigma)\sigma\epsilon- \frac{s}{t2^k} \right|.
\end{align*}
Let $s \gg 1$; if $s \ll 2^k t/ \epsilon$, then 
\begin{align*}
|F^{(m)}(\sigma)| \sleq 1, \; \; \forall m \geq 1, \; \; |\tilde{P_2}'(\sigma)| &\sleq \epsilon, \; |\tilde{P_2}''(\sigma)| &\sleq \epsilon^{1/2}, \; |\tilde{P_2}'''(\sigma)| &\sleq \epsilon, \;  |\tilde{P_2}^{(m)}(\sigma)| &\stackrel{\epsilon\leq1}{\sleq} 1, \; \forall m \geq 4
\end{align*}
while for $s \gg 2^kt/ \epsilon$
\begin{align*}
|F^{(m)}(\sigma)| \sleq 1, \; \; \forall m \geq 1, \; \; |\tilde{P_2}^{(m)}(\sigma)| &\stackrel{\epsilon\leq1}{\sleq} 1, \; \forall m \geq 1.
\end{align*}
Integrating by parts we get
\begin{align*}
A^{(1)}_k(s) &= 2^k(2^{2k}t)^{-N} K_d\pi e^{iR_1s} \int_{1/2}^2 e^{ it 2^{2k}\tilde{P_2}(\sigma) } \frac{\di}{\di \sigma} \left( \frac{1}{ \tilde{P_2}'(\sigma) } \cdots \frac{\di}{\di \sigma} \left( \frac{1}{ \tilde{P_2}'(\sigma) } \frac{\di}{\di \sigma} \left( \frac{ F(\sigma) }{ \tilde{P_2}'(\sigma) } \right) \right) \right) \di\sigma.
\end{align*}
Therefore
\begin{align} \label{estAk1}
|A_k^{(1)}(s)| &\sleq 2^k(2^{2k}t)^{-N}.
\end{align}
If $s \sim 2^kt/\epsilon$, we apply Van der Corput Lemma,
\begin{align*}
|A_k^{(1)}(s)| &\sleq 2^k(2^{2k}t)^{-1/2} \int_{1/2}^2 |\d_\sigma F(\sigma)|\di \sigma \\
&\stackrel{\eqref{FTradial2}}{\sleq} 2^k(2^{2k}t)^{-1/2} s^{-(d-1)/2} \sleq 2^k(2^{2k}t)^{-d/2} \epsilon^{(d-1)/2}.
\end{align*}
Moreover, we can check that $|A_k^{(1)}(s)| \sleq 2^k$; hence, for $s \gg 1$
\begin{align} \label{esta1high}
|A_k^{(1)}(s)| &\stackrel{\epsilon \leq 1}{\sleq} 2^k \min( 1, (2^{2k}t)^{-d/2} ).
\end{align}

If $s \sleq 1$, we rewrite $A_k^{(1)}$ in the following form
\begin{align*}
A_k^{(1)}(s) &= 2^k K_d \pi e^{-iR_1 s} \int_{1/2}^2 F_1(\sigma) e^{itP_2(R_1+2^k\sigma)} \di\sigma, \\
F_1(\sigma) &:= (R_1+2^k\sigma)^{d-1} \tilde\rho(R_1+2^k\sigma)\phi_0(\sigma) \bar h((R_1+2^k\sigma)s) e^{-i2^k\sigma s}.
\end{align*}
Again integrating by parts, we obtain
\begin{align} \label{esta1low}
|A_k^{(1)}(s)| &\sleq 2^k \min( 1, (2^{2k}t)^{-d/2} ). \\
\end{align}

Now we estimate $A_k^{(2)}$.  We notice that $R_2 \in \text{supp}(\phi_k(R_1-\cdot))$ if and only if $k \in \{-2,-1\}$; when $k \notin \{-2,-1\}$ one can repeat the above argument and show that
\begin{align} \label{esta2first}
|A_k^{(2)}(s)| &\sleq 2^k \min( 1, (2^{2k}t)^{-d/2} ).
\end{align}
Let $k \in \{-2,-1\}$. If $s \ll t$ or $s \gg t$ we have by integration by parts that
\begin{align*}
|A_k^{(2)}(s)| &\sleq \min( 1, t^{-N} ), \; \; \forall N \in \N.
\end{align*}
On the other hand, if $s \sim t$ we can use Van der Corput Lemma and obtain
\begin{align*} 
|A_k^{(2)}(s)| &\sleq t^{-1/3} s^{-(d-1)/2} \sleq t^{-\frac{d}{2}+\frac{1}{6}}.
\end{align*}
Therefore, for $k \in \{-2,-1\}$ we have 
\begin{align} \label{esta2second}
|A_k^{(2)}(s)| &\sleq \min( 1,t^{-\frac{d}{2}+\frac{1}{6}} ).
\end{align}
Combining \eqref{esta2first} and \eqref{esta2second} we can deduce that
\begin{align} \label{esta2}
|A_k^{(2)}(s)| &\sleq 2^k \min( 1,(2^{2k}t)^{-\frac{d}{2}+\frac{1}{6}} ).
\end{align}
If we sum up all the $A_k$ for $k \leq K+2$ we finally conclude that 
for any $d \geq 2$
\begin{align}
\| P_= \, \cU_2(t)\psi_0 \|_{L^\infty} &\sleq c \min(|t|^{-d/2},|t|^{-d/2+1/6}) \|\psi_0\|_{L^1}. \label{dispestmed}
\end{align}

\begin{remark} \label{genrcase}
In the general case $r>2$, we have to determine critical points for the polynomial
\begin{align} \label{genpol}
P_r(R) &= \sum_{j=1}^r (-1)^{j+1} \epsilon^{j-1}R^{2j},
\end{align}
namely the roots of the polynomial
\begin{align} \label{genpold}
P'_r(R) &= \sum_{j=1}^r (-1)^{j+1} \epsilon^{j-1} 2j R^{2j-1} = R \left(\sum_{j=1}^r (-1)^{j+1} \epsilon^{j-1} 2j R^{2(j-1)}\right).
\end{align}
Besides the trivial value $R=0$, which we deal as in the case $r=2$, one should rely on lower and upper bounds to determine the other (if any) real roots. For a lower bound, we rely on a well-known corollary of Rouché theorem from complex analysis, and we obtain that the other roots satisfy
\begin{align*}
R &\geq \frac{2}{\max\left( 2,\sum_{j=1}^r 2j\epsilon^{j-1} \right)}  \\
&\geq \frac{2}{\max\left( 2,2r \sum_{j=0}^{r-1} \epsilon^j \right)} \\
&\stackrel{\epsilon \leq 1/2}{\geq} \frac{2}{\max(2,4r\epsilon)} 
\stackrel{\epsilon \ll 1/(2r)}{\geq} 1.
\end{align*}
For what concerns an upper bound, we exploit an old result by Fujiwara (\cite{fujiwara1916obere}), and we get that the roots satisfy
\begin{align*}
R &\leq \max_{1 \leq j \leq r-1} \left( 2(r-1) \frac{2j\epsilon^{j-1}}{2r\epsilon^{r-1}} \right)^{ \frac{1}{2(j-1)} } \\
&\leq 2(r-1) \max_{1 \leq j \leq r-1} \left(\frac{j}{r}\right)^{ \frac{1}{2(j-1)} }  \epsilon^{ \frac{j-r}{2(j-1)} } \\
&\stackrel{\epsilon \leq 1}{\leq} K_r \epsilon^{-1/2}
\end{align*}
for some $K_r>0$. 

Hence, in the case $r>2$, if $\epsilon$ sufficiently small (depending on $r$), then the polynomial $P'_r$ has critical points (apart from 0) which have modulus between $1$ and $\cO(\epsilon^{-1/2})$ (a similar argument works also for the polynomial $P''_r$), and this affects the medium-frequency decay of $\cU_r(t)$. In any case, we can deal with this problem as in the case $r=2$, and we get
\begin{align}
\| P_< \, \cU_r(t)\psi_0 \|_{L^\infty} &\sleq |t|^{-d/2} \|\psi_0\|_{L^1}, \label{dispestrlow} \\
\| P_= \, \cU_r(t)\psi_0 \|_{L^\infty} &\sleq c \min(|t|^{-d/2},|t|^{-d/2+1/6}) \|\psi_0\|_{L^1}, \label{dispestrmed} \\
\| P_> \, \cU_r(t)\psi_0 \|_{L^\infty} &\sleq c^{d/2}|t|^{-\frac{d}{2r}} \|\psi_0\|_{L^1}, \; 0 < |t| \sleq c^{2(r-1)}. \label{dispestrhigh}
\end{align}

\end{remark}

\subsubsection{Smoothing estimates} \label{smoothsubsec}

As already pointed out, one needs smoothing estimates to ensure the well-posedness of Eq. \eqref{highordschr} because of the presence of derivatives in the nonlinearity. Again, we first consider the case $r=2$, and then we mention the results for $r>2$.

\begin{proposition} \label{smoothprop1}
For any $k=(k_1,\ldots,k_d) \in \Z^d$ with $|k_i|=|k|_\infty$ and $|k_i|\sgeq c$
\begin{align} \label{smoothest1}
\left\| \square_k D_{x_i}^{3/2}\cU_2(t)\psi_0 \right\|_{ L^{\infty,2}_{x_i;(x_j)_{j\neq i},t} } &\sleq c \|\square_k \psi_0\|_{L^2}.
\end{align}
\end{proposition}

\begin{proof}
It suffices to consider the case $i=1$. 
For convenience, we write $\bar z=(z_1,\ldots,z_d)$. Then,
\begin{align*}
\left\| \square_k D_{x_i}^{3/2}\cU_2(t)\psi_0 \right\|_{ L^{\infty,2}_{x_i;(x_j)_{j\neq i},t} } &= \left\| \int \sigma_k(\xi)|\xi_1|^{3/2} e^{itP_2(|\xi|)} \cF(\psi_0)(\xi) e^{ix_1\xi_1} \di\xi_1 \right\|_{ L^\infty_{x_1} L^2_{\bar\xi,t} } \\
&\sleq \left\| \int \eta_{k_1}(\xi_1)|\xi_1|^{3/2} e^{itP_2(|\xi|)} \cF(\psi_0)(\xi) e^{ix_1\xi_1} \di\xi_1 \right\|_{ L^\infty_{x_1} L^2_{\bar\xi,t} } =: L.
\end{align*}
Now, we estimate $L$: if $k_1 \sgeq c$, then $\xi_1>0$ for $\xi \in \text{supp}(\eta_{k_1})$. Hence, by changing variable, $\theta = P_2(|\xi|)$, we get
\begin{align*}
L &\sleq \left\| \int \eta_{k_1}( \xi_1(\theta) )\xi_1(\theta)^{3/2} e^{it\theta} \cF(\psi_0)(\xi(\theta)) e^{ix_1\xi_1(\theta)} \; \frac{1}{2} \xi_1^{-1}(\theta) \left( 2 \frac{|\xi|^2}{c^2}-1 \right)^{-1} \right\|_{ L^\infty_{x_1} L^2_{\bar\xi,t} } \\
&\sleq \left\| \eta_{k_1}( \xi_1(\theta) )\xi_1(\theta)^{1/2} \cF(\psi_0)(\xi(\theta)) \left( 2 \frac{|\xi|^2}{c^2}-1 \right)^{-1} \right\|_{ L^2_{\theta} L^2_{\bar\xi} } \\
&\sleq \left\| \eta_{k_1}( \xi_1 )\xi_1^{1/2} \cF(\psi_0)(\xi) \left( 2 \frac{|\xi|^2}{c^2}-1 \right)^{-1} \left( 2 \frac{|\xi|^2}{c^2}-1 \right)^{1/2} \xi_1^{1/2} \right\|_{ L^2_{\xi} } \\
&= \left\| \eta_{k_1}( \xi_1 )\xi_1 \cF(\psi_0)(\xi) \left( 2 \frac{|\xi|^2}{c^2}-1 \right)^{-1/2} \right\|_{ L^2_{\xi} }  \; \sleq \; c \|\psi_0\|_{L^2}. \\
\end{align*}
The proof for the case $k_1 \sleq -c$ is similar.
\end{proof}

By duality we have the following

\begin{proposition} \label{smoothprop2}
For any $k=(k_1,\ldots,k_d) \in \Z^d$ with $|k_i|=|k|_\infty$ and $|k_i|\sgeq c$
\begin{align} \label{smoothest2}
\left\| \square_k \d^2_{x_i} \cA_2f \right\|_{ L^\infty_t L^2_x } &\sleq c \|\square_k D_i^{1/2} f\|_{ L^{1,2}_{x_i;(x_j)_{j \neq i},t} }.
\end{align}
\end{proposition}

Now consider the inhomogeneous Cauchy problem
\begin{align} \label{inhom}
-i\psi_t &= A_{c,2}\psi + f(t,x), \; \; \psi(0,x)=0.
\end{align}

\begin{proposition} \label{smoothprop3}
For any $k=(k_1,\ldots,k_d) \in \Z^d$ with $|k_i|=|k|_\infty$ and $|k_i|\sgeq c$
\begin{align} \label{smoothest3}
\left\| \square_k \d^2_{x_i} \psi \right\|_{ L^{\infty,2}_{x_i;(x_j)_{j \neq i},t} } &\sleq \|\square_k f\|_{ L^{1,2}_{x_i;(x_j)_{j \neq i},t} }.
\end{align}
\end{proposition}

\begin{proof}
It suffices to consider $i=1$. We write 
\begin{align*}
\psi &= \cF^{-1}_{\tau,\xi} \frac{1}{\tau-c^2-P_2(|\xi|)}(\cF_{t,x}f)(\tau,\xi).
\end{align*}
We have
\begin{align} \label{d2u}
\d^2_{x_i}\psi &= \cF^{-1}_{\tau,\xi} \frac{\xi_1^2}{P_2(|\xi|)+c^2-\tau}\cF_{t,x}f.
\end{align}
We want to show that
\begin{align*}
\left\| \cF^{-1}_{\tau,\xi} \frac{\eta_{k_1}(\xi_1)\xi_1^2}{P_2(|\xi|)+c^2-\tau}\cF_{t,x}f \right\|_{L^\infty_{x_1} L^2_{\bar\xi,t}} &\sleq \left\| \cF^{-1}_{\xi_1} \eta_{k_1}(\xi_1) \cF_{x_1}f \right\|_{L^1_{x_1} L^2_{\bar\xi,t}},
\end{align*}
which, by Young's inequality, is equivalent to show that
\begin{align} \label{est3thesis}
\sup_{x_1,\tau,\xi_j \; (j \neq 1)} \left| \cF^{-1}_{\xi_1} \frac{\sigma_k(\xi)\xi_1^2}{P_2(|\xi|)+c^2-\tau} \right| &\sleq 1.
\end{align}

We prove \eqref{est3thesis}: first, notice that when $|k_1|=|k|_\infty$, then $|\xi_1| \sim |\xi|_\infty$ for $\xi \in \text{supp}(\sigma_k)$. 
We split the argument according to the cases $\tau-c^2>0$ and $\tau-c^2\leq0$. 
In the case $\tau-c^2 > 0$
\begin{align*}
\sup_{x_1,\tau,\xi_j \; (j \neq 1)} \left| \cF^{-1}_{\xi_1} \frac{\sigma_k(\xi)\xi_1^2}{P_2(|\xi|)+c^2-\tau} \right| & \sleq \left| \int_{k_1-3/4}^{k_1+3/4} \frac{c^2}{\xi_1^2} \di\xi_1 \right| \; \stackrel{|k_1|\sgeq c}{\sleq} 1. 
\end{align*}
\\

When $\tau-c^2 \leq 0$ we set $\tau_2:=\tau_2(c)= c \left( \sqrt{\frac{5}{4}-\frac{\tau}{c^2}} - \frac{1}{2} \right) > 0$, in order to write
\begin{align*}
P_2(|\xi|)+c^2-\tau &= \left( \frac{|\xi|^2}{c}+\tau_2 \right) \left( -\frac{|\xi|^2}{c}+\tau_2+c \right).
\end{align*}
Hence
\begin{align}
 \cF^{-1}_{\xi_1} \frac{\sigma_k(\xi)\xi_1^2}{P_2(|\xi|)+c^2-\tau} &= \cF^{-1}_{\xi_1} \frac{\sigma_k(\xi)\xi_1^2}{ \left( \frac{|\xi|^2}{c}+\tau_2 \right) \left( -\frac{|\xi|^2}{c}+\tau_2+c \right) } \nonumber \\
&= \cF^{-1}_{\xi_1} \frac{\sigma_k(\xi)\xi_1^2}{ \left( \frac{\xi_1^2}{c}+\frac{|\bar\xi|^2}{c}+\tau_2 \right) \left( -\frac{\xi_1^2}{c}-\frac{|\bar\xi|^2}{c}+\tau_2+c \right) }. \label{smooth3symb}
\end{align}
When $|\bar\xi|^2 \geq c(\tau_2+c)$, we can treat the problem as before.

Next, we consider the case $|\bar\xi|^2 < c(\tau_2+c)$. Let
\begin{align*}
A^2 &:= A(\bar\xi,\tau,c)^2 = \frac{|\bar\xi|^2}{c}+\tau_2, \\
B^2 &:= B(\bar\xi,\tau,c)^2 = -\left( \frac{|\bar\xi|^2}{c}-\tau_2-c \right),
\end{align*}
then
\begin{align}
&\cF^{-1}_{\xi_1} \frac{\eta_{k_1}(\xi_1)\xi_1^2}{ \left( \frac{|\xi|^2}{c}+\tau_2 \right) \left( -\frac{|\xi|^2}{c}+\tau_2+c \right) } \nonumber \\
&= \cF^{-1}_{\xi_1} \frac{\xi_1}{ \frac{\xi_1^2}{c}+A^2 } \frac{\xi_1}{ B^2-\frac{\xi_1^2}{c} } \eta_{k_1}(\xi_1) \nonumber \\
&= \frac{c^{1/2}}{2} \cF^{-1}_{\xi_1} \frac{\xi_1}{ \frac{\xi_1^2}{c}+A^2 } \left( \frac{1}{  B -\frac{\xi_1}{c^{1/2}} } - \frac{1}{  B +\frac{\xi_1}{c^{1/2}} } \right) \eta_{k_1}(\xi_1) \nonumber \\
&=: I + II. \nonumber
\end{align}
We estimate only $I$, as the argument of $II$ is similar. First we write
\begin{align*}
I &= -\frac{c}{2} \cF^{-1}_{\xi_1} \frac{\eta_{k_1}(\xi_1)}{ \frac{\xi_1^2}{c}+A^2  } + \frac{c}{2} \cF^{-1}_{\xi_1} \frac{\eta_{k_1}(\xi_1)B}{ ( B -\frac{\xi_1}{c^{1/2}} ) (  \frac{\xi_1^2}{c}+A^2 ) } := I_1+I_2.
\end{align*}
Since $\cF^{-1}_{\xi_1}(1/\xi_1)$ is the function $sgn(\xi_1)$, we have that 
$I_1$ is bounded uniformly with respect to $c$. For $I_2$, it suffices to show 
\begin{align*}
cB \sup_{x_1} \left| \cF^{-1}_{\xi_1} \frac{1}{ ( B -\frac{\xi_1}{c^{1/2}} ) (  \frac{\xi_1^2}{c}+A^2 ) } \right| &\sleq 1.
\end{align*}
Since $|\cF(e^{-|\cdot|})(\xi)| \sleq \frac{1}{1+|\xi|^2}$,
\begin{align*}
cB \left\| \cF^{-1}_{\xi_1} \frac{1}{ ( B -\frac{\xi_1}{c^{1/2}} ) (  \frac{\xi_1^2}{c}+A^2 ) } \right\|_{L^\infty_{x_1}} &\sleq cB \left\| \cF^{-1}_{\xi_1} \frac{1}{ B -\frac{\xi_1}{c^{1/2}} } \right\|_{L^\infty_{x_1}} \left\| \cF^{-1}_{\xi_1} \frac{1}{ \frac{\xi_1^2}{c}+A^2 } \right\|_{L^1_{x_1}} \\
&\sleq \frac{c^2B}{A^2} \left\| \cF^{-1}_{\xi_1} \frac{1}{ B -\frac{\xi_1}{c^{1/2}} } \right\|_{L^\infty_{x_1}} \left\| \cF^{-1}_{\xi_1} \frac{1}{ \xi_1^2A^{-2} + 1} \right\|_{L^1_{x_1}} \\
&\sleq B \left\| \cF^{-1}_{\xi_1} \frac{1}{ B -\frac{\xi_1}{c^{1/2}} } \right\|_{L^\infty_{x_1}} \cdot \frac{c^2}{A^2} \left\| A e^{-A|x_1|} \right\|_{L^1_{x_1}} \; \sleq \; 1.
\end{align*}

Finally, we observe that in general the solution $\psi$ of \eqref{inhom} may not vanish at $t=0$. However, by Parseval identity
\begin{align*}
\psi(0,x) &= \psi(t,x)_{|t=0} = K \int_I \cU_2(s) \cF(f)(s,x) \di s,
\end{align*}
for some $K>0$, and if we combine it with \eqref{smoothest2}, we have that 
$\square_k\cU_2(t) \di^2_{x_1}\psi(0,x) \in L^2$. Hence, by \eqref{smoothest1} 
\begin{align}
\tilde\psi(t) &:= \psi(t)-\cU_2(t)\psi(0,\cdot) = i \int_I \cU_2(t-\tau)f(\tau)\di\tau
\end{align}
is the solution of \eqref{inhom}, and it satisfies \eqref{smoothest3}.
\end{proof}

\begin{lemma}
For any $\sigma \in \R$ and $k \in \Z^d$ with $|k_i| \geq 4$,
\begin{align} \label{danest}
\| \square_k D^\sigma_{x_i}\psi \|_{ L^{p_1,p_2}_{x_1;(x_j)_{j \neq 1},t} } &\sleq \la k_i \ra^\sigma \| \square_k \psi \|_{ L^{p_1,p_2}_{x_1;(x_j)_{j \neq 1},t} }. 
\end{align}
If we replace $D^\sigma_{x_i}$ by $\d^\sigma_{x_i}$, the above inequality holds 
for all $k \in \Z^d$.
\end{lemma}
\begin{proof}
See the proof of Lemma 3.4 in \cite{wang2009global}. 
One can check that both sides of \eqref{danest} are equivalent for $|k_i|\geq4$.
\end{proof}

By combining \eqref{smoothest3}, \eqref{smoothest2} and \eqref{danest} we obtain

\begin{proposition}
For any $k=(k_1,\ldots,k_d) \in \Z^d$ with $|k_i|=|k|_\infty \sgeq c$ we have
\begin{align} 
\left\| \square_k \d_{x_i}^2\cA_2f \right\|_{ L^{\infty,2}_{x_i;(x_j)_{j\neq i},t} } &\sleq \|\square_k f\|_ {L^{1,2}_{x_i;(x_j)_{j\neq i},t} }, \label{smoothestfin1} \\
\left\| \square_k \d^2_{x_i} \cA_2f \right\|_{ L^\infty_t L^2_x } &\sleq c \la|k_i|\ra^{1/2} \|\square_k f\|_{ L^{1,2}_{x_i;(x_j)_{j \neq i},t} }. \label{smoothestfin2} 
\end{align}
\end{proposition}

\begin{remark}
For the case $r>2$ we replace \eqref{smoothest1}, \eqref{smoothest2}, 
\eqref{smoothest3}, \eqref{smoothestfin1} and \eqref{smoothestfin2} with
\begin{align} 
\left\| \square_k D_{x_i}^{r-1/2}\cU_r(t)\psi_0 \right\|_{ L^{\infty,2}_{x_i;(x_j)_{j\neq i},t} } &\sleq c^{r-1} \|\square_k \psi_0\|_{L^2}, \label{smoothestr1} \\
\left\| \square_k \d^r_{x_i} \cA_rf \right\|_{ L^\infty_t L^2_x } &\sleq c^{r-1} \|\square_k D_i^{1/2} f\|_{ L^{1,2}_{x_i;(x_j)_{j \neq i},t} }, \label{smoothestr2} \\
\left\| \square_k \d^{2(r-1)}_{x_i} \psi \right\|_{ L^{\infty,2}_{x_i;(x_j)_{j \neq i},t} } &\sleq \|\square_k f\|_{ L^{1,2}_{x_i;(x_j)_{j \neq i},t} }, \label{smoothestr3} \\
\left\| \square_k \d_{x_i}^{2(r-1)}\cA_rf \right\|_{ L^{\infty,2}_{x_i;(x_j)_{j\neq i},t} } &\sleq \|\square_k f\|_ {L^{1,2}_{x_i;(x_j)_{j\neq i},t} }, \label{smoothestfinr1} \\
\left\| \square_k \d^{2(r-1)}_{x_i} \cA_rf \right\|_{ L^\infty_t L^2_x } &\sleq c^{r-1} \la|k_i|\ra^{r-3/2} \|\square_k f\|_{ L^{1,2}_{x_i;(x_j)_{j \neq i},t} }. \label{smoothestfinr2}
\end{align}
\end{remark}

\begin{remark}
We point out the fact that we have worked out smoothing estimates only in the 
higher frequencies. As in \cite{ruzhansky2016global}, only these smoothing 
estimates are needed in order to discuss the well-posedness of \eqref{highordschr}.
\end{remark}

\subsubsection{Strichartz estimates} \label{strestsubsec}

By exploiting \eqref{strhighschr} we can deduce Strichartz estimates for 
solutions of \eqref{highordschr} combined with $\square_k$-decomposition 
operators. 

\begin{proposition}
Let $r \geq 1$, $d \geq 2$, $c \geq 1$, $t>0$. 
Let $(p,q)$ and $(a,b)$ be order-$r$ admissible pairs. 
Then for any $0 < T \sleq c^{2(r-1)}$ and for any $k \in \Z^d$ 
with $|k| \sgeq K$ ($K=K(c)$ is defined in \eqref{threshold})
\begin{align}
\| \square_k \cU_r(t)\phi_0 \|_{L^p([0,T])L^q(\R^d)} &\sleq 
c^{ d \left( 1-\frac{1}{r} \right) \left(\frac{1}{2} -\frac{1}{q} \right) } \|\square_k\phi_0\|_{L^2(\R^d)} \nonumber \\ 
&= c^{ \left(1-\frac{1}{r}\right) \frac{2r}{p} } \|\square_k\phi_0\|_{L^2(\R^d)}, \label{strhighschrdec} \\
\left\| \square_k \int_0^t \cU_r(t-\tau)\phi(\tau) \di\tau \right\|_{L^p([0,T])L^q(\R^d)} &\sleq c^{\left(1-\frac{1}{r}\right)2r\left(\frac{1}{p}+\frac{1}{a}\right)} \|\square_k\phi\|_{L^{a'}([0,T])L^{b'}(\R^d)}. \label{retstrhighschrdec} 
\end{align}
\end{proposition}

Furthermore, by \eqref{locdispest} we have that
\begin{align*}
\| \square_k \cU_r(t) \|_{L^1 \to L^\infty} &\sleq c^{d \left(1-\frac{1}{r}\right)} \la t \ra^{-d/(2r)}, \; 0 < |t| \sleq c^{2(r-1)},
\end{align*}

and by following closely the argument in Section 5 of \cite{wang2007global} we 
can deduce

\begin{proposition} 
Let $r \geq 1$, $d \geq 2$, $c \geq 1$.
Let $(p,q)$ be a Schr\"odinger admissible pair, then
\begin{align}
\| \cU_r(t)\psi_0 \|_{l^{1,s}_\square(L^p_t([0,T]) L^q_x)} &\sleq c^{\left(1-\frac{1}{r}\right)\frac{2r}{p}} \|\psi_0\|_{M^s_{2,1}}, \; 0 < T \sleq c^{2(r-1)},  \label{strestfrdec} \\
\| \cA_r f \|_{l^{1,s}_\square(L^p_t([0,T]) L^q_x) \cap l^{1,s}_\square(L^\infty_t([0,T]) L^2_x)} &\sleq c^{\left(1-\frac{1}{r}\right)\frac{4r}{p}} \|f\|_{l^{1,s}_\square (L^{p'}([0,T])L^{q'}(\R^d))}. \label{retstrestfrdec}
\end{align}
\end{proposition}

\subsubsection{Maximal function estimates} \label{maxfsubsec}

In this subsection we study the maximal function estimates for the semigroup 
$\cU_r(t)$ and the integral operator $\int_0^t \cU_r(t-\tau) \cdot \di\tau$ 
in anisotropic Lebesgue spaces. To do this, we will need the time decay 
properties proved in Sec. \ref{tdecsubsec}. As always, we first prove results 
for the case $r=2$, and then we write the modification for the general case.

\begin{lemma}
\begin{enumerate}
\item Let $q \geq 2$, $\frac{8}{d} < q \leq + \infty$ and $k \in \Z^d$ 
with $|k| \sgeq K(c)$, then 
\begin{align} \label{maxfest1h}
\| \square_k \, \cU_2(t)\psi_0 \|_{L^{q,\infty}_{x_i;(x_j)_{j \neq i},t}} &\sleq c^{d/2} \la k \ra^{1/q} \|\square_k\psi_0\|_{L^2}, \; 0 < |t| \sleq c^2, \; \forall i=1,\ldots,d.
\end{align}

\item Let $q \geq 2$, $\frac{4}{d} < q \leq + \infty$ and $k \in \Z^d$ 
with $|k| \sleq K(c)$, then 
\begin{align} \label{maxfest1l}
\| \square_k \, \cU_2(t)\psi_0 \|_{L^{q,\infty}_{x_i;(x_j)_{j \neq i},t}} &\sleq c \la k \ra^{1/q} \|\square_k\psi_0\|_{L^2}, \; \forall i=1,\ldots,d.
\end{align}
\end{enumerate}
\end{lemma}

\begin{proof}
Clearly it suffices to show the thesis for $i=1$; recall that for any 
$x=(x_1,\ldots,x_d) \in \R^d$ we denote $\bar x = (x_2,\ldots,x_d)$. 
By a standard $TT^\star$ argument, \eqref{maxfest1h} is equivalent to
\begin{align} \label{maxfest1th}
\left\| \int_{\R^d} e^{i \la x,\xi \ra} e^{it(c^2+P_2(|\xi|)) }\sigma_k(\xi)\di\xi \right\|_{L^{q/2,\infty}_{x_1;\bar x,t } } &\sleq \la k \ra^{2/q}.
\end{align}
If $|k| \sgeq K(c)$, then
\begin{align}
\| \cF^{-1}e^{it(c^2+P_2(|\xi|))}\sigma_k(\xi) \|_{L^\infty_x} &\stackrel{\eqref{dispesthigh}}{\sleq} c^{d/2} \la k \ra^{-d} |t|^{-d/4}, \; 0 < |t| \sleq c^2; \label{estdec1}
\end{align}
on the other hand
\begin{align}
\| \square_k \cU_2(t)\cF^{-1}\sigma_k \|_{L^\infty_{t,x}} &\sleq \| \square_k \cU_2(t)\cF^{-1}\sigma_k \|_{L^\infty_t L^2_x} \sleq 1. \label{estdec2}
\end{align}
If we combine \eqref{estdec1} and \eqref{estdec2}, we obtain
\begin{align}
|\square_k\cU_2(t)\cF^{-1}\sigma_k| &\sleq c^{d/2} (1+\la k \ra^4 |t|)^{-d/4}, \; 0 < |t| \sleq c^2. \label{estdec3}
\end{align}

Now, if $|x_1| \sgeq 1+|t|\la k \ra^5$, by integrating by parts we get
\begin{align} \label{estdec4}
| \square_k \cU_2(t)\cF^{-1}\sigma_k | &\sleq c^{d/2} \la x_1 \ra^{-2} . 
\end{align}
If $|x_1| \sleq 1+|t|\la k \ra^5$, by \eqref{estdec3} we can deduce
\begin{align} \label{estdec5}
| \square_k \cU_2(t)\cF^{-1}\sigma_k | &\sleq c^{d/2} (1+|x_1|\la k \ra^{-1})^{-d/4}
\end{align}
Combining \eqref{estdec4} and \eqref{estdec5} we have
\begin{align} \label{estdec6}
\sup_{\bar x,t}|\square_k \cU_2(t)\cF^{-1}\sigma_k| &\sleq c^{d/2} \la x_1 \ra^{-2} + c^{d/2} (1+|x_1|\la k \ra^{-1})^{-d/4}, 
\end{align}
from which, by taking the $L^{q/2}_{x_1}$ norm on both sides, we obtain 
\eqref{maxfest1th}. The proof for the case $|k| \sleq K(c)$ is similar.
\end{proof}

\begin{lemma}
Let $q \geq 2$, $\frac{8}{d} < q \leq + \infty$ and $k \in \Z^d$ 
with $|k_i| \sgeq K(c)^2$, then 
\begin{align} \label{maxfest2h}
\| \square_k \, \cA_2 f \|_{L^{q,\infty}_{x_i;(x_j)_{j \neq i},t}} &\sleq c^{d/2} \la k_i \ra^{-3/2+1/q} \|\square_kf\|_{L^{1,2}_{x_i;(x_j)_{j \neq i},t}}, \; 0 < |t| \sleq c^2, \; \forall i=1,\ldots,d.
\end{align}
\end{lemma}

\begin{proof}
It suffices to prove the case $i=1$. 
Recall that the solution of \eqref{inhom} is of the form 
\begin{align*}
\psi &= \cF^{-1}_{\tau,\xi} \frac{1}{c^2+P_2(|\xi|)-\tau}\cF_{t,x}f,
\end{align*}
hence its frequency localization can be written as
\begin{align*}
\square_k \, \psi &= \cF^{-1}_{\tau,\xi} \frac{1}{c^2+P_2(|\xi|)-\tau}(\cF_{t,x}\square_k \, f)(\tau,\xi).
\end{align*}
For convenience, we introduce the following regions
\begin{align*}
\E_1 &= \{\tau-c^2 \leq -c^2/4 \}, \\
\E_2 &= \left\{-c^2/4 \leq \tau-c^2 \leq |\bar\xi|^2 \left( -\frac{|\bar\xi|^2}{c^2}+1 \right) \right\}, \\
\E_3 &= \left\{ \tau-c^2 \geq |\bar\xi|^2 \left( -\frac{|\bar\xi|^2}{c^2}+1 \right) \right\},
\end{align*}
and we make the following decomposition
\begin{align} \label{maxfdec2}
c^2+P_2(|\xi|)-\tau &= 
\begin{cases}
\left( \frac{|\xi|^2}{c} + \tau_2(c,\tau) \right) \left( -\frac{\xi_1}{c^{1/2}} + a \right) \left( \frac{\xi_1}{c^{1/2}} + a \right)  & (\bar\xi,\tau) \in \E_1, \\
\left( \frac{|\xi|^2}{c} + \tau_2(c,\tau) \right) \left( -\frac{|\xi|^2}{c} + \tau_2(c,\tau)+c \right), & (\bar\xi,\tau) \in \E_2, \\
- \left( \frac{|\xi|^2}{c} - \frac{c}{2} \right)^2 + \left(\frac{5}{4} c^2- \tau\right), & (\bar\xi,\tau) \in \E_3, \\
\end{cases}
\end{align}
where $a=a(c,\bar\xi,\tau):= ( \tau_2(c,\tau)- |\bar\xi|^2/c + c)^{1/2}$.
We denote 
\begin{align*}
\square_k \, \psi_i &= \cF^{-1}_{\tau,\xi} \frac{ \chi_{\E_i}(\bar\xi,\tau) }{c^2+P_2(|\xi|)-\tau}(\cF_{t,x}\square_k \, f)(\tau,\xi), \; i=1,2,3.
\end{align*}
First, we estimate $\square_k \, \psi_1$. Set $\tilde\eta_{k_1}(\xi_1) = \sum_{|l| \leq 10} \eta_{k_1+l}(\xi_1)$. First we notice that 
\begin{align*}
\frac{ \chi_{\E_1}(\bar\xi,\tau) }{c^2+P_2(|\xi|)-\tau} &= \frac{ \chi_{\E_1}(\bar\xi,\tau) }{ (2\tau_2(c,\tau)+c) \left( \frac{|\xi|^2}{c} + \tau_2(c,\tau) \right) }  \\
&\; \; + \frac{ \chi_{\E_1}(\bar\xi,\tau) }{ 2a(2\tau_2(c,\tau)+c) } \left( \frac{1}{ -\frac{\xi_1}{c^{1/2}} + a } + \frac{1}{ \frac{\xi_1}{c^{1/2}} + a } \right) \\
&=: \sum_{j=1}^3 A_j(c,\xi,\tau). 
\end{align*}
According to the above decomposition, we can rewrite $\square_k \, \psi_1$ as
\begin{align*}
& \square_k \, \psi_1 \\
&= \cF^{-1}_{\tau,\xi} \frac{ \chi_{\E_1}(\bar\xi,\tau) \tilde\eta_{k_1}(ac^{1/2}) }{c^2+P_2(|\xi|)-\tau}(\cF_{t,x}\square_k \, f)(\tau,\xi) + \cF^{-1}_{\tau,\xi} \frac{ \chi_{\E_1}(\bar\xi,\tau) (1-\tilde\eta_{k_1}(ac^{1/2})) }{c^2+P_2(|\xi|)-\tau}(\cF_{t,x}\square_k \, f)(\tau,\xi) \\
&= \sum_{j=1}^3 \cF^{-1}_{\tau,\xi} \chi_{\E_1}(\bar\xi,\tau) A_j(\xi,\tau) \tilde\eta_{k_1}(ac^{1/2}) (\cF_{t,x}\square_k \, f)(\tau,\xi) \\
&\; \; \; + \cF^{-1}_{\tau,\xi} \frac{ \chi_{\E_1}(\bar\xi,\tau) (1-\tilde\eta_{k_1}(ac^{1/2})) }{c^2+P_2(|\xi|)-\tau}(\cF_{t,x}\square_k \, f)(\tau,\xi) \\
&=: I + II + III + IV.
\end{align*}
 
\emph{Case $k_1 \sgeq K(c)^2$}: first, we estimate II. Let $\tilde\sigma_k$ be as in \eqref{decop}, then
\begin{align*}
II &= \int_{I \times \R^d} \frac{ e^{it\tau+i \la \bar x,\bar\xi \ra} \chi_{\E_1}(\bar\xi,\tau) }{2a(2\tau_2(c,\tau)+c)} \tilde\sigma_{\bar k}(\bar\xi) \tilde\eta_{k_1}(ac^{1/2}) \widehat{\square_k f(y_1,\cdot)}(\bar\xi,\tau) c^{1/2} e^{i(x_1-y_1)ac^{1/2}} \; sgn(x_1-y_1) \di\bar\xi \di y_1 \di\tau.
\end{align*}

By changing variable, $\xi_1=c^{1/2} a(c,\bar\xi,\tau)$, and by setting 
$\tilde\rho_k(\xi) =  \tilde\sigma_{\bar k}(\bar\xi) \tilde\eta_{k_1}(\xi_1)$, 
we obtain
\begin{align*}
|II| &\sleq \left| \int \di y_1 \; sgn(x_1-y_1) \int e^{it(c^2+P_2(|\xi|))} e^{i(x_1-y_1)\xi_1 + i \la \bar x,\bar\xi \ra} \tilde\rho_k(\xi) \widehat{\square_k f(y_1,\cdot)}(c^2+P_2(|\xi|),\tau) \di\xi \right|,
\end{align*}
and by applying \eqref{maxfest1h} we get
\begin{align}
\| II \|_{L^{q,\infty}_{x_1;\bar x,t}} &\sleq \int \di y_1 \left\| \int e^{it(c^2+P_2(|\xi|))} e^{i(x_1-y_1)\xi_1 + i \la \bar x,\bar\xi \ra} \tilde\rho_k(\xi) \widehat{\square_k f(y_1,\cdot)}(c^2+P_2(|\xi|),\tau) \di\xi \right\|_{L^{q,\infty}_{x_1;\bar x,t}} \nonumber \\
&\sleq c^{d/2} \la k_1 \ra^{1/q} \int \| \tilde\rho_k(\xi) \widehat{\square_k f(y_1,\cdot)}(c^2+P_2(|\xi|),\tau) \|_{L^2_\xi} \di y_1 \nonumber \\
&\stackrel{\eqref{smoothest2},\eqref{danest}}{\sleq} c \, c^{d/2} \la k_1 \ra^{1/q-3/2} \|\square_k \, f\|_{L^{1,2}_{x_1;\bar x,t}}. \label{boundIII}
\end{align}

Since $k_1>0$, $III$ has the same upper bound as in \eqref{boundIII}. 

Now we estimate $IV$: first notice that 
\begin{align*}
IV &= \int \di y_1 \int e^{it\tau + i \la \bar x, \bar\xi \ra} \tilde\sigma_{\bar k}(\bar\xi) \widehat{\square_k f(y_1,\cdot)}(\bar\xi,\tau) \; K(x_1-y_1,a,\bar\xi) \di\bar\xi, \\
K(x_1,a,\bar\xi) &= \chi_{\E_1}(\bar\xi,\tau) (1-\tilde\eta_{k_1}(ac^{1/2})) \int \frac{\sum_{|l| \leq 1} \eta_{k_1+l}(\xi_1) e^{ix_1 \xi_1} }{c^2+P_2(|\xi|)-\tau} \di\xi_1.
\end{align*}
By Young's inequality for convolutions, H\"older's inequality and Minkowski's 
inequality we have 
\begin{align*}
\| IV \|_{L^{q,\infty}_{x_1;\bar x,t}} &\leq \left\| \int \| \tilde\sigma_{\bar k}(\bar\xi) \widehat{\square_k f(y_1,\cdot)}(\bar\xi,\tau) \; K(x_1-y_1,a,\bar\xi) \|_{L^1_{\bar\xi,\tau}} \di y_1 \right\|_{L^q_{x_1}} \\
&\leq \| \square_k \, f \|_{L^{1,2}_{x_1;\bar x,t}} \| \tilde\sigma_{\bar k}(\bar\xi)K(x_1,a,\bar\xi) \|_{L^{q,2}_{x_1;\bar\xi,\tau}} \\
&\sleq \| \square_k \, f \|_{L^{1,2}_{x_1;\bar x,t}} | \tilde\sigma_{\bar k}(\bar\xi)K(x_1,a,\bar\xi) \|_{L^\infty_{\bar\xi} L^2_\tau L^q_{x_1}}.
\end{align*}

Integrating by parts it follows that
\begin{align} \label{estK}
\| \tilde\sigma_{\bar k}(\bar\xi) K(x_1,a,\xi) \|_{L^\infty_{\bar\xi} L^2_\tau L^q_{x_1}} &\sleq \sup_{|\xi-k|_\infty \leq 3} \sum_{j=0}^1 \| \chi_{\E_1}(\bar\xi,\tau) (1-\tilde\eta_{k_1}(ac^{1/2})) \d^j_{\xi_1}(c^2+P_2(|\xi|)-\tau)^{-1} \|_{L^2_\tau}.
\end{align}

Noticing that $|\xi - ac^{1/2}| \geq c^{1/2} \geq 1$ in the support set of 
$(1-\tilde\eta_{k_1}(ac^{1/2})) \chi_{|\xi_1-k_1| \leq 3}\d^j_{\xi_1}(c^2+P_2(|\xi|)-\tau)^{-1}$ 
we can deduce from \eqref{maxfdec2} that there is no singularity 
if we integrate \eqref{estK}, and this gives 
\begin{align*}
\| \tilde\sigma_{\bar k}(\bar\xi) K(x_1,a,\bar\xi) \|_{L^\infty_{\bar\xi} L^2_\tau L^q_{x_1}} &\sleq c^{1/2} |k_1|^{-3/2}.
\end{align*}

Now we estimate $I$: we begin by setting
\begin{align} \label{estI1}
J(x_1,a,\bar\xi) &= \chi_{\E_1}(\bar\xi,\tau) \tilde\eta_{k_1}(ac^{1/2}) \int \frac{ \sum_{|l| \leq 1} \eta_{k+l}(\xi_1)e^{ix_1\xi_1} }{ (2\tau_2(c,\tau)+c)\left( \frac{|xi|^2}{c}+\tau_2(c,\tau) \right) } \di\xi_1.
\end{align}
One can check that
\begin{align*}
 I &= \int \di y_1 \int e^{it\tau + i \la \bar x, \bar\xi \ra} \tilde\sigma_{\bar k}(\bar\xi) \widehat{\square_k \, f(y_1,\cdot)}(\bar\xi,\tau) J(x_1-y_1,a,\bar\xi) \di\bar\xi \di\tau.
\end{align*}
Similar to the estimate of $IV$, by Young's, H\"older's and Minkowski's inequalities we obtain
\begin{align*}
\|I\|_{L^{q,\infty}_{x_1;\bar x,t}} &\sleq c^{1/2} \|\square_k f\|_{L^{1,2}_{x_1;\bar x,t}} \|\tilde\sigma_k(\bar\xi)J(x_1,a,\bar\xi)\|_{L^\infty_{\bar\xi} L^2_\tau L^q_{x_1}}.
\end{align*}
By integration by parts we get
\begin{align*}
|J(x_1,a,\xi)| \sleq \frac{ \chi_{\E_1}(\bar\xi,\tau) \tilde\eta_{k_1}(ac^{1/2}) }{(2\tau_2(c,\tau)+c)(1+|x_1|) } \sum_{j=0}^1 \int_{|\xi_1-k_1| \leq 3} \left|\d^j_{\xi_1} \left( \frac{|xi|^2}{c}+\tau_2(c,\tau) \right)^{-1} \right| \di\xi_1.
\end{align*}
Therefore 
\begin{align} \label{estJ}
\|\tilde\sigma_k(\bar\xi)J(x_1,a,\bar\xi)\|_{L^\infty_{\bar\xi} L^2_\tau L^q_{x_1}} &\sleq \sup_{|\xi-k|_\infty \leq 3} \sum_{j=0}^1 \left\| \frac{ \chi_{\E_1}(\bar\xi,\tau) \tilde\eta_{k_1}(ac^{1/2}) }{(2\tau_2(c,\tau)+c)(1+|x_1|) } \sum_{j=0}^1 \left|\d^j_{\xi_1} \left( \frac{|xi|^2}{c}+\tau_2(c,\tau) \right)^{-1} \right| \right\|_{L^2_\tau},
\end{align}
and noticing that $|ac^{1/2}-k_1| \leq 20$ in the support set of $\tilde\eta_{k_1}(ac^{1/2})$, we can deduce that $2\tau_2(c,\tau)+c \sgeq k_1^2$, and finally we obtain
\begin{align} \label{estJ2}
\|\tilde\sigma_k(\bar\xi)J(x_1,a,\bar\xi)\|_{L^\infty_{\bar\xi} L^2_\tau L^q_{x_1}} &\sleq |k_1|^{-2}.
\end{align}

The proof for the case $k \sleq - K(c)^2$ is similar. 
Furthermore, in the estimate of $\square_k \, \psi_2$ and $\square_k \, \psi_3$ 
we can check that there is no singularity in $(c^2+P_2(|\xi|)-\tau)^{-1}$ for 
$|\xi_1| \geq c^{1/2}$ and $(\bar\xi,\tau) \in \E_2 \cup \E_3$. Hence, one can 
argue as in \eqref{estI1}-\eqref{estJ2} and conclude.
\end{proof}

In the last Lemma we proved that 
$\square_k \cA_2: L^{1,2}_{x_1,(x_j)_{j \neq 2},t} \to L^{2,\infty}_{x_1,(x_j)_{j \neq 2},t}$. 
In the next Lemma we show that 
$\square_k \cA_2: L^{1,2}_{x_2,(x_j)_{j \neq 1},t} \to L^{2,\infty}_{x_1,(x_j)_{j \neq 2},t}$.

\begin{lemma}
Let $q \geq 2$, $\frac{8}{d} < q \leq + \infty$, $k \in \Z^d$ 
with $|k_i| \sgeq c$ and $h,i \in \{1,\ldots,d\}$ with $h \neq i$, then 
\begin{align} \label{maxfest3h}
\| \square_k \, \cA_2 f \|_{L^{q,\infty}_{x_h;(x_j)_{j \neq h},t}} &\sleq c^{1+d/2} \la k_i \ra^{-3/2+1/q} \|\square_kf\|_{L^{1,2}_{x_i;(x_j)_{j \neq i},t}}, \; 0 < |t| \sleq c^2.
\end{align}
\end{lemma} 

\begin{proof}
It clearly suffices to consider the case $h=1$, $i=2$ and $k_2 \sgeq c$. 
The proof goes along the same line of that of \eqref{maxfest2h}, and we will 
only prove in detail the parts that are different. For convenience, we denote 
$\tilde\xi=(\xi_1,\xi_3,\ldots,\xi_d)$. We introduce the following regions 
\begin{align*}
\F_1 &= \{\tau-c^2 \leq -c^2/4 \}, \\
\F_2 &= \left\{-c^2/4 \leq \tau-c^2 \leq |\tilde\xi|^2 \left( -\frac{|\tilde\xi|^2}{c^2}+1 \right) \right\}, \\
\F_3 &= \left\{ \tau-c^2 \geq |\tilde\xi|^2 \left( -\frac{|\tilde\xi|^2}{c^2}+1 \right) \right\},
\end{align*}
and we make the following decomposition
\begin{align} \label{maxfdec3}
c^2+P_2(|\xi|)-\tau &= 
\begin{cases}
\left( \frac{|\xi|^2}{c} + \tau_2(c,\tau) \right) \left( -\frac{\xi_2}{c^{1/2}} + a \right) \left( \frac{\xi_2}{c^{1/2}} + a \right)  & (\tilde\xi,\tau) \in \F_1, \\
\left( \frac{|\xi|^2}{c} + \tau_2(c,\tau) \right) \left( -\frac{|\xi|^2}{c} + \tau_2(c,\tau)+c \right), & (\tilde\xi,\tau) \in \F_2, \\
- \left( \frac{|\xi|^2}{c} - \frac{c}{2} \right)^2 + \left(\frac{5}{4} c^2- \tau\right), & (\tilde\xi,\tau) \in \F_3, \\
\end{cases}
\end{align}
where $b=b(c,\tilde\xi,\tau):= ( \tau_2(c,\tau)- |\tilde\xi|^2/c + c)^{1/2}$, 
$\tau_2(c,\tau)= c \left( \sqrt{\frac{5}{4}-\frac{\tau}{c^2}} - \frac{1}{2} \right)$.
We denote 
\begin{align*}
\square_k \, \tilde\psi_i &= \cF^{-1}_{\tau,\xi} \frac{ \chi_{\F_i}(\tilde\xi,\tau) }{c^2+P_2(|\xi|)-\tau}(\cF_{t,x}\square_k \, f)(\tau,\xi), \; i=1,2,3.
\end{align*}
We estimate $\square_k\,\tilde\psi_1$, since by definition of the regions $\F_i$ the estimate of the other terms follow more easily, like in the last Lemma.

Set $\tilde\eta_{k_2}(\xi_2) = \sum_{|l| \leq 10} \eta_{k_2+l}(\xi_2)$. First we notice that 
\begin{align*}
\frac{ \chi_{\F_1}(\tilde\xi,\tau) }{c^2+P_2(|\xi|)-\tau} &= \frac{ \chi_{\F_1}(\tilde\xi,\tau) }{ (2\tau_2(c,\tau)+c) \left( \frac{|\xi|^2}{c} + \tau_2(c,\tau) \right) }  \\
&\; \; + \frac{ \chi_{\F_1}(\tilde\xi,\tau) }{ 2b(2\tau_2(c,\tau)+c) } \left( \frac{1}{ -\frac{\xi_1}{c^{1/2}} + b } + \frac{1}{ \frac{\xi_1}{c^{1/2}} + b } \right) \\
&=: \sum_{j=1}^3 B_j(c,\xi,\tau). 
\end{align*}
According to the above decomposition, we can rewrite $\square_k \, \psi_1$ as
\begin{align*}
& \square_k \, \tilde\psi_1 \\
&= \sum_{j=1}^3 \cF^{-1}_{\tau,\xi} \chi_{\F_1}(\tilde\xi,\tau) B_j(\xi,\tau) \tilde\eta_{k_2}(bc^{1/2}) (\cF_{t,x}\square_k \, f)(\tau,\xi) \\
&\; \; \; + \cF^{-1}_{\tau,\xi} \frac{ \chi_{\F_1}(\tilde\xi,\tau) (1-\tilde\eta_{k_2}(bc^{1/2})) }{c^2+P_2(|\xi|)-\tau}(\cF_{t,x}\square_k \, f)(\tau,\xi) \\
&=: I + II + III + IV.
\end{align*}

The estimates of $II$ and $III$ follow in the same way as for \eqref{maxfest2h} 
by exchanging the roles of $\xi_1$ and $\xi_2$. Now we estimate $I$: set 
\begin{align} \label{defm}
m(\xi,\tau) &= \frac{ \chi_{\E_1}(\tilde\xi,\tau)\tilde\sigma_{\bar k}(\tilde\xi) \tilde\eta_{k_2}(bc^{1/2} }{2b(2\tau_2(c,\tau)+c)},
\end{align}
and notice that $2\tau_2(c,\tau)+c \sgeq k_2^2$ in the support set of $m$; 
hence, for sufficiently large $c$, we have
\begin{align*}
m(\xi,\tau) &\sleq \frac{ \chi_{\F_1} \tilde\sigma_k(\xi) }{k_2^4},
\end{align*}
and therefore
\begin{align} \label{estm1}
\|m\|_{L^1_{\xi_2}L^2_{\xi_3,\ldots,\xi_d,\tau} L^{2q/(q-2)}_{\xi_1}} \sleq |k_2|^{-2}.
\end{align}
Now, since by Young's, H\"older's and Minkowski's inequalities we have 
\begin{align}
\|\cF^{-1}_{\xi,\tau}m(\xi,\tau)(\cF_{t,x}\square_k \, f)\|_{L^{q,\infty}_{x_1;\bar x,t}} &\sleq \|\cF^{-1}_{\xi,\tau}m(\xi,\tau)(\cF_{t,x}\square_k \, f)\|_{L^q_{x_1} L^1_{\tilde\xi,\tau}} \nonumber \\
&\sleq \|m(\xi,\tau)(\cF_{t,x}\square_k \, f)\|_{L^1_{\tilde\xi,\tau} L^{q'}_{\xi_1}} \nonumber \\
&\sleq \|m\|_{L^1_{\xi_2}L^2_{\xi_3,\ldots,\xi_d,\tau} L^{2q/(q-2)}_{\xi_1}} \|\cF_{t,x}\square_k \, f\|_{L^\infty_{\xi_2} L^2_{(\xi_j)_{j\neq2},\tau}} \nonumber \\
&\sleq \|m\|_{L^1_{\xi_2}L^2_{\xi_3,\ldots,\xi_d,\tau} L^{2q/(q-2)}_{\xi_1}} \|\square_k \, f\|_{L^1_{x_2} L^2_{(x_j)_{j\neq2},t}}, \label{estm2}
\end{align}
we can deduce that
\begin{align} \label{estIf}
\|I\|_{L^{q,\infty}_{x_1;\bar x,t}} &\sleq |k_2|^{-2} \|\square_k \, f\|_{ L^1_{x_1} L^2_{(x_j)_{j \neq 2},t} }.
\end{align}

Now we estimate $IV$: set 
\begin{align} \label{mkBkdef}
m_k(\xi,\tau): = \frac{ \chi_{\F_3}(\tilde\xi,\tau) \tilde\sigma_k(\xi) (1-\tilde\eta_{k_2}(b)) }{ c^2+P_2(|\xi|)-\tau }, \; &\; \cM_k(f) := \cF^{-1}_{\tau,\xi}m_k(\xi,\tau)(\cF_{t,x}f), \\
\end{align}
and notice that $\cM_k(f)$ is the solution of the inhomogeneous equation
\begin{align*}
-i\psi_t &= A_{c,2}\psi - \cF^{-1}_{\tau,\xi}m_k(\xi,\tau)(c^2+P_2(|\xi|)-\tau)(\cF_{t,x}f).
\end{align*}
Applying \eqref{smoothestfin2} (recall that $k_2 \sgeq c$), we have 
\begin{align}
\|\cM_k(f)\|_{L^\infty_{x,t}} &\sleq c^{3/2} \|\cM_k(f)\|_{L^\infty_t L^2_x} \nonumber \\
&\sleq c^{d/2} c |k_2|^{-3/2} \|f\|_{ L^1_{x_2} L^2_{(x_j)_{j \neq 2},t} } \nonumber \\
&= c^{\frac{d}{2} +1} |k_2|^{-3/2} \|f\|_{ L^1_{x_2} L^2_{(x_j)_{j \neq 2},t} }. \label{estmk0}
\end{align}
% c^{d/2 + r-1} for r>1
Next, for $(\xi,\tau) \in \text{supp}(m_k)$,
\begin{align} \label{estmk1}
|c^2+P_2(|\xi|)-\tau| &\stackrel{(\tilde\xi,\tau) \in \F_3}{\sgeq} c^{-1} \la k \ra^2 |k_2|.
\end{align}
% c^{-(r-1)}
By the definition of $b$ we have that for $\frac{|\tilde\xi|}{c} \sleq \frac{1}{2} (c+\tau_2(c,\tau))$
\begin{align} \label{estdiff1}
|c^2+P_2(|\xi|)-\tau| &\sgeq ( c+\tau_2(c,\tau) )^{3/2}, \; (\xi,\tau) \in \text{supp}(m_k), 
\end{align}
while for $\frac{|\tilde\xi|}{c} \sgeq \frac{1}{2} (c+\tau_2(c,\tau))$ 
we can exploit the fact that $|k|_\infty=|k_2| \sgeq c$ to obtain again that
\begin{align} \label{estdiff2}
|c^2+P_2(|\xi|)-\tau| &\sgeq ( c+\tau_2(c,\tau) )^{3/2}, \; (\xi,\tau) \in \text{supp}(m_k), 
\end{align}
and by combining \eqref{mkBkdef} with \eqref{estdiff1}-\eqref{estdiff2} we obtain
\begin{align} \label{estmk2}
m_k(\xi,\tau) \sleq c \frac{\chi_{\tau \geq \frac{3}{4}c^2}\tilde\sigma_k(\xi) }{(|k_2|^{2}+c+\tau_2(c,\tau))^{3/2} },
\end{align}
which gives
\begin{align} \label{estmk3}
\|m_k\|_{L^1_{\xi_2}L^2_{\xi_3,\ldots,\xi_d,\tau}L^\infty_{\xi_1}} &\sleq c |k_2|^{-1}.
\end{align}
Therefore, from\eqref{mkBkdef} and \eqref{estm2} we can deduce
\begin{align} \label{estmk4}
\|\cM_k(f)\|_{L^{2,\infty}_{x_1;\bar x,t}} &\sleq c |k_2|^{-1} \|f\|_{L^1_{x_2}L^2_{(x_j)_{j \neq 2},t}}.
\end{align}
For any $q \geq 2$ we obtain by interpolation between \eqref{estmk0} and 
\eqref{estmk4} 
\begin{align} \label{estmk5}
\|\cM_k(f)\|_{L^{q,\infty}_{x_1;\bar x,t}} &\sleq c^{1+d\left( \frac{1}{2}-\frac{1}{q} \right)} |k_2|^{-3/2+1/q} \|f\|_{L^1_{x_2}L^2_{(x_j)_{j \neq 2},t}},
\end{align}
%c^{r-1+d\left( \frac{1}{2}-\frac{1}{q} \right)} for r>2
and replacing $f$ by $\square_k \, f$ in \eqref{estmk5}, we finally obtain
\begin{align*}
\|IV\|_{L^{q,\infty}_{x_1;\bar x,t}} &\sleq c^{1+d\left( \frac{1}{2}-\frac{1}{q} \right)} |k_2|^{-3/2+1/q} \|f\|_{L^1_{x_2}L^2_{(x_j)_{j \neq 2},t}}.
\end{align*}

\end{proof}

If we collect \eqref{maxfest2h} and \eqref{maxfest3h}, we can deduce 

\begin{lemma}
Let $q \geq 2$, $\frac{8}{d} < q \leq + \infty$, $k \in \Z^d$ 
with $|k_i| \sgeq c$ and $h,i \in \{1,\ldots,d\}$, then 
\begin{align} \label{maxfest4h}
\| \square_k \, \d^2_{x_i}\cA_2 f \|_{L^{q,\infty}_{x_h;(x_j)_{j \neq h},t}} &\sleq c^{1+d/2} \la k_i \ra^{1/2+1/q} \|\square_kf\|_{L^{1,2}_{x_i;(x_j)_{j \neq i},t}}, \; 0 < |t| \sleq c^2.
\end{align}
\end{lemma} 

\begin{remark}
In the general case $r>2$ we have
\begin{enumerate}
\item Let $q \geq 2$, $\frac{4r}{d} < q \leq + \infty$ and $k \in \Z^d$ 
with $|k| \sgeq K(c)$, then 
\begin{align} \label{maxfest1hr}
\| \square_k \, \cU_r(t)\psi_0 \|_{L^{q,\infty}_{x_i;(x_j)_{j \neq i},t}} &\sleq c^{d\left(1-\frac{1}{r}\right)} \la k \ra^{1/q} \|\square_k\psi_0\|_{L^2}, \; 0 < |t| \sleq c^{2(r-1)}, \; \forall i=1,\ldots,d.
\end{align}

\item Let $q \geq 2$, $\frac{4}{d} < q \leq + \infty$ and $k \in \Z^d$ 
with $|k| \sleq K(c)$, then 
\begin{align} \label{maxfest1lr}
\| \square_k \, \cU_r(t)\psi_0 \|_{L^{q,\infty}_{x_i;(x_j)_{j \neq i},t}} &\sleq c \la k \ra^{1/q} \|\square_k\psi_0\|_{L^2}, \; \forall i=1,\ldots,d.
\end{align}

\item Let $q \geq 2$, $\frac{4r}{d} < q \leq + \infty$ and $k \in \Z^d$ 
with $|k_i| \sgeq K(c)^2$ and $i \in \{1,\ldots,d\}$, then 
\begin{align} \label{maxfest2hr}
\| \square_k \, \cA_r f \|_{L^{q,\infty}_{x_i;(x_j)_{j \neq i},t}} &\sleq c^{r-1+d\left(1-\frac{1}{r}\right)} \la k_i \ra^{-r+1/2+1/q} \|\square_kf\|_{L^{1,2}_{x_i;(x_j)_{j \neq i},t}}, \; 0 < |t| \sleq c^{2(r-1)}.
\end{align}

\item Let $q \geq 2$, $\frac{4r}{d} < q \leq + \infty$, $k \in \Z^d$ 
with $|k_i| \sgeq c$ and $h,i \in \{1,\ldots,d\}$ with $h \neq i$, then 
\begin{align} \label{maxfest3hr}
\| \square_k \, \cA_r f \|_{L^{q,\infty}_{x_h;(x_j)_{j \neq h},t}} &\sleq c^{r-1+d\left(1-\frac{1}{r}\right)} \la k_i \ra^{-r+1/2+1/q} \|\square_kf\|_{L^{1,2}_{x_i;(x_j)_{j \neq i},t}}, \; 0 < |t| \sleq c^{2(r-1)}.
\end{align}

\item Let $q \geq 2$, $\frac{4r}{d} < q \leq + \infty$, $k \in \Z^d$ 
with $|k_i| \sgeq c$ and $h,i \in \{1,\ldots,d\}$, then
\begin{align} \label{maxfest4hr}
\| \square_k \, \d^{2(r-1)}_{x_i}\cA_r f \|_{L^{q,\infty}_{x_h;(x_j)_{j \neq h},t}} &\sleq c^{r-1+d\left(1-\frac{1}{r}\right)} \la k_i \ra^{r-3/2+1/q} \|\square_kf\|_{L^{1,2}_{x_i;(x_j)_{j \neq i},t}}, \; 0 < |t| \sleq c^{2(r-1)}.
\end{align}

\end{enumerate}
\end{remark}

\subsubsection{Proof of the local well-posedness} \label{LWPproofsubsec}

In this subsection we use smoothing estimates, Strichartz estimates and 
maximal funtion estimates in order to prove Proposition \ref{LWPhighordschr}. 
In order to do so, it seems necessary to estimate norms in which partial 
derivatives and anisotropic Lebesgue spaces have different directions, 
for example $\|\d^2_{x_1} \square_k \, \cA f\|_{L^{2,\infty}_{x_2;(x_j)_{j \neq 2},t}}$ with 
$|k|_\infty = |k_3|$. As usual, we show results for the case $r=2$, and then 
we point out the modifications for the case $r>2$.

\begin{lemma}
Let $i,l,m \in \{1,\ldots,d\}$, $1 \leq p,q, \leq +\infty$. 
Assume that $k=(k_1,\ldots,k_d)$ with $|k|_\infty = |k_m| \sgeq c$, then
\begin{align} \label{estparder}
\| \square_k \, \d^2_{x_l}f\|_{L^{p,q}_{x_i;(x_j)_{j \neq i},t}} &\sleq \| \square_k \, \d^2_{x_m}f\|_{L^{p,q}_{x_i;(x_j)_{j \neq i},t}}.
\end{align}
\end{lemma}

\begin{proof}
\begin{align*}
\| \square_k \, \d^2_{x_l}f\|_{L^{p,q}_{x_i;(x_j)_{j \neq i},t}} &\sleq \sum_{|h_l|_\infty,|h_m|_\infty \leq 1} \left\| \cF^{-1}_{\xi_l,\xi_m} \left( \left( \frac{\xi_l}{\xi_m} \right)^2 \eta_{k_l+h_l}(\xi_l)\eta_{k_m+l_m}(\xi_m) \right) \right\|_{L^1(\R^2)} \\
&\times \| \square_k \, \d^2_{x_m}f\|_{L^{p,q}_{x_i;(x_j)_{j \neq i},t}} \\
&\sleq \| \square_k \, \d^2_{x_m}f\|_{L^{p,q}_{x_i;(x_j)_{j \neq i},t}}.
\end{align*}
\end{proof}

\begin{lemma}
\begin{enumerate}
\item Let $(a,b)$ be order-$2$ admissible, $i \in \{1,\ldots,d\}$, $q \geq 2$, 
$\frac{8}{d} < q <+\infty$ and $k \in \Z^d$ with $|k|_\infty \sgeq K(c)$, then
\begin{align} \label{strest2}
\| \square_k \; \d^\alpha_{x_i} \cA_2 f \|_{L^{q,\infty}_{x_i;(x_j)_{j \neq i},t}} \sleq c^{\frac{d}{2} + \frac{2}{a}}  \la |k|_\infty \ra^{\alpha+1/q} \|\square_k \, f\|_{L^{a'}_t L^{b'}_x}, \; \; 0 < |t| \sleq c^2.
\end{align}

\item Let $(a,b)$ be Schr\"odinger admissible, $i \in \{1,\ldots,d\}$, then
\begin{align}
\| \square_k \d^2_{x_i}\cA_2f \|_{L^a_t L^b_x} &\sleq c^{1+4/p} \la |k|_\infty \ra^{1/2} \|\square_k f \|_{L^{1,2}_{x_i;(x_j)_{j \neq i},t}}, \; 0 < |t| \sleq c^2 \label{auxstr1} \\
\| \square_k \d^2_{x_i}\cA_2f \|_{L^{\infty,2}_{x_i;(x_j)_{j \neq i},t}} &\sleq c^{1+4/p} \la |k|_\infty \ra^{1/2} \|\square_k f \|_{L^{a'}_t L^{b'}_x}, \; 0 < |t| \sleq c^2 \label{auxstr2}.
\end{align}
\end{enumerate}
\end{lemma}

\begin{proof}
Denote
\begin{align*}
\cL_k(f,\psi) &= \int \left( \square_k \, \int \cU_2(t-\tau)f(\tau)\di\tau, \psi(t) \right) \di t.
\end{align*}
By dualiy and the maximal function estimate \eqref{maxfest1h}
\begin{align*}
|\cL_k(f,\psi)| &\leq \| \square_k \, f\|_{L^{q'}_{x_1} L^1_{\bar x,t}} \sum_{|l|_\infty \leq 1} \left\| \square_{k+l} \int \cU_2(t-\tau)\psi(t) \di t \right\|_{L^q_{x_1} L^\infty_{\bar x,t}} \\
&\leq \| \square_k \, f\|_{L^{q'}_{x_1} L^1_{\bar x,t}} \sum_{|l|_\infty \leq 1} \int \| \square_{k+l} \cU_2(t-\tau)\psi(t) \di t \|_{L^q_{x_1} L^\infty_{\bar x,t}} \\
&\stackrel{\eqref{maxfest1h}}{\leq}  c^{d/2} \la k \ra^{1/q} \| \square_k \, f\|_{L^{q'}_{x_1} L^1_{\bar x,t}} \|\psi\|_{L^1_t L^2_x},
\end{align*}
so by duality we obtain
\begin{align} \label{auxest}
\left\| \square_k \, \int \cU_2(t-\tau)f(\tau)\di\tau \right\|_{L^\infty_tL^2_x} &\sleq c^{d/2} \la k \ra^{1/q} \| \square_k \, f\|_{L^{q'}_{x_1} L^1_{\bar x,t}}.
\end{align}
Therefore, by duality, Strichartz estimates \eqref{retstrhighschrdec} and 
\eqref{auxest}
\begin{align}
|\cL_k(f,\psi)| &\leq \left\| \square_k \, \int \cU_2(-\tau)f(\tau)\di\tau \right\|_{L^2_x} \left\| \square_k \, \int \cU_2(-t)\psi(t)\di t \right\|_{L^2_x} \nonumber \\
&\sleq c^{d/2} \la k \ra^{1/q}  \|f\|_{L^{q'}_{x_1} L^1_{\bar x,t}} \; c^{(1-1/r)2r/a} \|\square_k \, \psi\|_{L^{a'}_t L^{b'}_x},
\end{align}
which implies \eqref{strest2} for $q>2$ or $a>2$. 
In the case $a=q=2$, \eqref{strest2} can be directly deduced from \eqref{maxfest1h}.
Furthermore, by \eqref{retstrestfrdec}, \eqref{danest} and \eqref{smoothestfin2} we get
\begin{align}
\cL_k(\d^2_{x_i}f,\psi) &\sleq c^{1+4/p} \la|k|_\infty\ra^{1/2} \|\square_k f\|_{ L^{1,2}_{x_i;(x_j)_{j \neq i},t} } \; c^{4/p} \|\psi\|_{L^{a'}_t L^{b'}_x},
\end{align}
and we can deduce \eqref{auxstr1}; by exchanging $f$ and $\psi$, we get \eqref{auxstr2}.
\end{proof}

We now summarize the results we will use in order to prove the local 
well-posedness of \eqref{highordschr}: we omit the proof, it follows 
from the results of the previous subsections, together with \eqref{estparder}.

\begin{proposition} \label{auxprop}
Let $d \geq 2$, $8/d \leq p < + \infty$, $2 \leq q <+\infty$, $q > 8/d$, 
$k \in \Z^d$ with $|k|_\infty = |k_i| \sgeq c$, $h,i,l \in \{1,\ldots,d\}$. Then
\begin{align}
\left\| \square_k D_{x_i}^{3/2}\cU_2(t)\psi_0 \right\|_{ L^{\infty,2}_{x_i;(x_j)_{j\neq i},t} } &\sleq c \|\square_k \psi_0\|_{L^2}, \label{smoothest1def} \\
\| \square_k \, \cU_2(t)\psi_0 \|_{L^{q,\infty}_{x_i;(x_j)_{j \neq i},t}} &\sleq c^{d/2} \la k \ra^{1/q} \|\square_k\psi_0\|_{L^2}, \; 0 < |t| \sleq c^2, \label{maxfest1hdef} \\
\| \square_k \cU_2(t)\phi_0 \|_{L^\infty_t L^2_x \cap L^{2+p}_x} &\sleq c^{\frac{4}{p(p+2)}} \|\square_k\phi_0\|_{L^2}, \; 0 < |t| \sleq c^2 \label{strhighschrdecdef} \\
\| \square_k \d^2_{x_l} \cA_2 f \|_{ L^{\infty,2}_{x_i;(x_j)_{j\neq i},t} } &\sleq \|\square_k \, f\|_{ L^{1,2}_{x_i;(x_j)_{j\neq i},t} }, \label{smoothestfin1def} \\
\| \square_k \, \d^2_{x_l}\cA_2 f \|_{L^{q,\infty}_{x_h;(x_j)_{j \neq h},t}} &\sleq c^{1+d/2} \la k_i \ra^{1/2+1/q} \|\square_kf\|_{L^{1,2}_{x_i;(x_j)_{j \neq i},t}}, \; 0 < |t| \sleq c^2, \label{maxfest4hdef} \\
\| \square_k \d^2_{x_l}\cA_2f \|_{L^\infty_t L^2_x \cap L^{2+p}_{t,x}} &\sleq c^{1+\frac{4}{p(p+2)}} \la k_i \ra^{1/2} \|\square_k f \|_{L^{1,2}_{x_i;(x_j)_{j \neq i},t}}, \; 0 < |t| \sleq c^2 \label{auxstr1def} \\
\| \square_k \d^2_{x_l}\cA_2f \|_{L^{\infty,2}_{x_i;(x_j)_{j \neq i},t}} &\sleq c^{1+\frac{4}{p(p+2)}} \la k_i \ra^{1/2} \|\square_k f \|_{ L^{(2+p)/(1+p)}_{t,x} }, \; 0 < |t| \sleq c^2 \label{auxstr2def}, \\
\| \square_k \; \d^2_{x_l} \cA_2 f \|_{L^{q,\infty}_{x_i;(x_j)_{j \neq i},t}} &\sleq c^{\frac{d}{2} + \frac{2}{p+2} + \frac{4}{p(p+2)} }  \la k_i \ra^{2+1/q} \|\square_k \, f\|_{ L^{(2+p)/(1+p)}_{t,x} }, \; \; 0 < |t| \sleq c^2, \label{strest2def} \\
\| \square_k \; \cA_2 f \|_{L^\infty_tL^2_x \cap L^{2+p}_{t,x}} &\sleq c^{\frac{8}{p(p+2)}} \|\square_k \, f\|_{L^{(2+p)/(1+p)}_{t,x}}. \label{auxstr3def}
\end{align}
\end{proposition}

For the case $r>2$ we have the following results

\begin{remark}
\begin{enumerate}
\item Let $(a,b)$ be order-$r$ admissible, $i \in \{1,\ldots,d\}$, $q \geq 2$, 
$\frac{4r}{d} < q <+\infty$ and $k \in \Z^d$ with $|k|_\infty \sgeq K(c)$, then
\begin{align} \label{strest2r}
\| \square_k \; \d^\alpha_{x_i} \cA_r f \|_{L^{q,\infty}_{x_i;(x_j)_{j \neq i},t}} \sleq c^{d\left(1-\frac{1}{r}\right) + \left(1-\frac{1}{r}\right)\frac{2r}{a} } \la |k|_\infty \ra^{\alpha+1/q} \|\square_k \, f\|_{L^{a'}_t L^{b'}_x}, \; \; 0 < |t| \sleq c^{2(r-1)}.
\end{align}

\item Let $(a,b)$ be Schr\"odinger admissible, $i \in \{1,\ldots,d\}$, then
\begin{align}
\| \square_k \d^{2(r-1)}_{x_i}\cA_rf \|_{L^a_t L^b_x} &\sleq c^{r-1+2r/a} \la |k|_\infty \ra^{r-3/2} \|\square_k f \|_{L^{1,2}_{x_i;(x_j)_{j \neq i},t}}, \; \; 0 < |t| \sleq c^{2(r-1)} \label{auxstr1r} \\
\| \square_k \d^{2(r-1)}_{x_i}\cA_2f \|_{L^{\infty,2}_{x_i;(x_j)_{j \neq i},t}} &\sleq c^{r-1+2r/a} \la |k|_\infty \ra^{r-3/2} \|\square_k f \|_{L^{a'}_t L^{b'}_x}, \; \; 0 < |t| \sleq c^{2(r-1)} \label{auxstr2r}.
\end{align}
\end{enumerate}
\end{remark}

\begin{proposition} \label{auxpropr}
Let $d \geq 2$, $4r/d \leq p < + \infty$, $2 \leq q <+\infty$, $q > 4r/d$, 
$k \in \Z^d$ with $|k|_\infty = |k_i| \sgeq c$, $h,i,l \in \{1,\ldots,d\}$. Then
\begin{align}
\left\| \square_k D_{x_i}^{r-1/2}\cU_r(t)\psi_0 \right\|_{ L^{\infty,2}_{x_i;(x_j)_{j\neq i},t} } &\sleq c^{r-1} \|\square_k \psi_0\|_{L^2}, \label{smoothest1rdef} \\
\| \square_k \, \cU_r(t)\psi_0 \|_{L^{q,\infty}_{x_i;(x_j)_{j \neq i},t}} &\sleq c^{d\left(1-\frac{1}{r}\right)} \la k \ra^{1/q} \|\square_k\psi_0\|_{L^2}, \; 0 < |t| \sleq c^{2(r-1)}, \label{maxfest1hrdef} \\
\| \square_k \cU_r(t)\phi_0 \|_{L^\infty_t L^2_x \cap L^{2(r-1)+p}_x} &\sleq c^{\frac{4(r-1)^2}{p(p+2(r-1))}} \|\square_k\phi_0\|_{L^2(\R^d)}, \; 0 < |t| \sleq c^{2(r-1)} \label{strhighschrdecrdef} \\
\| \square_k \d^{2(r-1)}_{x_l} \cA_r f \|_{ L^{\infty,2}_{x_i;(x_j)_{j\neq i},t} } &\sleq \|\square_k \, f\|_{ L^{1,2}_{x_i;(x_j)_{j\neq i},t} }, \label{smoothestfin1rdef} \\
\| \square_k \, \d^{2(r-1)}_{x_l}\cA_r f \|_{L^{q,\infty}_{x_h;(x_j)_{j \neq h},t}} &\sleq c^{r-1+d\left(1-\frac{1}{r}\right)} \la k_i \ra^{r-3/2+1/q} \|\square_kf\|_{L^{1,2}_{x_i;(x_j)_{j \neq i},t}}, \; 0 < |t| \sleq c^{2(r-1)}, \label{maxfest4hrdef} \\
\| \square_k \d^{2(r-1)}_{x_l}\cA_rf \|_{L^\infty_t L^2_x \cap L^{2(r-1)+p}_{t,x}} &\sleq c^{ r-1+\frac{4(r-1)^2}{p(p+2(r-1))} } \la k_i \ra^{r-3/2} \|\square_k f \|_{L^{1,2}_{x_i;(x_j)_{j \neq i},t}}, \; 0 < |t| \sleq c^{2(r-1)} \label{auxstr1rdef} \\
\| \square_k \d^{2(r-1)}_{x_l}\cA_rf \|_{L^{\infty,2}_{x_i;(x_j)_{j \neq i},t}} &\sleq c^{ r-1+\frac{4(r-1)^2}{p(p+2(r-1))} } \la k_i \ra^{r-3/2} \|\square_k f \|_{ L^{ \frac{2(r-1)+p}{2r-1+p} }_{t,x} }, \; 0 < |t| \sleq c^{2(r-1)} \label{auxstr2rdef}, \\
\| \square_k \; \d^{2(r-1)}_{x_l} \cA_r f \|_{L^{q,\infty}_{x_i;(x_j)_{j \neq i},t}} &\sleq c^{ \frac{d}{2} + \frac{2r}{p+2(r-1)} + \frac{4(r-1)^2}{p(p+2(r-1))} }  \la k_i \ra^{2(r-1)+1/q} \|\square_k \, f\|_{ L^{ \frac{2(r-1)+p}{2r-1+p} }_{t,x} }, \nonumber \\
&\; \; \; \; \; 0 < |t| \sleq c^{2(r-1)}, \label{strest2rdef} \\
\| \square_k \; \cA_r f \|_{L^\infty_tL^2_x \cap L^{2(r-1)+p}_{t,x}} &\sleq c^{\frac{8(r-1)^2}{p(p+2(r-1))}} \|\square_k \, f\|_{L^{ \frac{2(r-1)+p}{2r-1+p}  }_{t,x}}. \label{auxstrrdef}
\end{align}
\end{proposition}

For convenience, we state some technical results related to nonlinear mapping 
estimates. For $i=1,\ldots,d$ and $N \in \N$ we set
\begin{align*}
\B_{i,1}^{(N)} &:= \{ (k^{(1)},\ldots,k^{(N)}) \in (\Z^d)^N : \max(|k^{(1)}_i|,\ldots,|k^{(N)}_i|) \sgeq c \}, \\
\B_{i,2}^{(N)} &:= \{ (k^{(1)},\ldots,k^{(N)}) \in (\Z^d)^N : \max(|k^{(1)}_i|,\ldots,|k^{(N)}_i|) \sleq c \}.
\end{align*}

\begin{lemma}
Let $s \geq 0$, $N \geq 3$, $i \in \{1,\ldots,d\}$, then
\begin{align}
&\left\| \sum_{\B_{i,1}^{(N)}} \square_{k^{(1)}} \, \psi_1 \cdots \square_{k^{(N)}} \, \psi_N \right\|_{ l^{1,s}_{\square,i,c}(L^{1,2}_{x_1,(x_j)_{j \neq 2},t}) } \nonumber \\
&\sleq \sum_{\alpha=1}^N \|\psi_\alpha\|_{ \cap_{h=1}^d l^{1,s}_{\square,h,c}(L^{\infty,2}_{x_h,(x_j)_{j \neq h},t}) } \prod_{ \substack{\beta=1,\ldots,d \\ \beta \neq \alpha} } \|\psi_\beta\|_{ \cap_{h=1}^d l^{1}_{\square}(L^{N-1\infty}_{x_h,(x_j)_{j \neq h},t}) } . \label{estnlin1}
\end{align}
\end{lemma}

\begin{proof}
See proof of Lemma 3.1 in \cite{ruzhansky2016global}.
\end{proof}

\begin{lemma}
Let $N \geq 1$ and $i \in \{1,\ldots,d\}$, and assume that 
$1 \leq p,q,p_1,q_1,\ldots,p_N,q_N \leq +\infty$ satisfy
\begin{align*}
\frac{1}{p} = \frac{1}{p_1} + \ldots + \frac{1}{p_N}, &\; \; \frac{1}{q} = \frac{1}{q_1} + \ldots + \frac{1}{q_N},
\end{align*}
then
\begin{align} \label{estnlin2}
\left\| \sum_{\B_{i,2}^{(N)}} \square_{k^{(1)}} \, \psi_1 \cdots \square_{k^{(N)}} \, \psi_N \right\|_{ l^1_{\square,i,c}(L^q_t L^p_x) } &\sleq c^d \, N^d \, \sum_{\B_{i,2}^{(N)}} \prod_{i=1}^N \|\square_{k^{(i)}} \, \psi_i\|_{L^{q_i}_t L^{p_i}_x}.
\end{align}
\end{lemma}

\begin{proof}
See proof of Lemma 3.3 in \cite{ruzhansky2016global}.
\end{proof}

\begin{lemma}
Let $s \geq 0$, $N \geq 1$ and $i \in \{1,\ldots,d\}$, and assume that 
$1 \leq p,q,p_1,q_1,\ldots,p_N,q_N \leq +\infty$ satisfy
\begin{align*}
\frac{1}{p} = \frac{1}{p_1} + \ldots + \frac{1}{p_N}, &\; \; \frac{1}{q} = \frac{1}{q_1} + \ldots + \frac{1}{q_N},
\end{align*}
then
\begin{align} \label{estnlin3}
\| \psi_1 \cdots \psi_N \|_{l^{1,s}_{\square}(L^p_t L^q_x)} &\sleq N^d \prod_{i=1}^N \| \psi_i \|_{l^{1,s}_{\square}(L^{p_i}_t L^{q_i}_x)}.
\end{align}
\end{lemma}

\begin{proof}
See proof of Lemma 8.2 in \cite{wang2007global}.
\end{proof}

\begin{proof}[Proof (Proposition \ref{LWPhighordschr}, part $(i)$, case $r=2$)] 
Since the nonlinearity contains terms of the form $(\d_x^\alpha \psi)^\beta$ 
with $|\alpha| \leq 2$, $|\beta| \geq m+1$, we introduce the space 
\begin{align*}
D &:= \{ \psi \in \cS': \|\psi\|_D := \sum_{|\alpha| \leq 2} \sum_{l=1}^3 \sum_{i,j=1}^d \rho_l^{(i)}(\d_{x_j}^\alpha\psi) \sleq c^{-\delta_0} \},
\end{align*}
where
\begin{align*}
\rho_1^{(i)}(\psi) &:= \|\psi\|_{l^{1,s-r+1/2+1/m}_{\square,i,c}(L^{\infty,2}_{x_i;(x_j)_{j \neq i},t})}, \\
\rho_2^{(i)}(\psi) &:= \|\psi\|_{l^{1,s}_{\square}(L^{m,\infty}_{x_i;(x_j)_{j \neq i},t})}, \\
\rho_3^{(i)}(\psi) &:= \|\psi\|_{l^{1,s+1/m}_{\square}(L^\infty_t L^2_x \cap L^{2+m}_{t,x})}.
\end{align*}
and for some $\delta_0>0$ that we will choose later. \\

Since $\|\psi\|_D=\|\bar{\psi}\|_D$, without loss of generality we can assume 
that the nonlinearity contain only terms of the form
\begin{align*}
\psi^{\beta_0}(\d_x^{\alpha_1}\psi)^{\beta_1}(\d_x^{\alpha_2}\psi)^{\beta_2} &=: \Psi_1\ldots\Psi_R,
\end{align*}
where $R:=|\beta|=\beta_0+|\beta_1|+|\beta_2|$, $|\alpha_i|=i$ ($i=1,2$). \\

To prove the first part of Proposition \ref{LWPhighordschr} we will show that the map 
\begin{align*}
\cF &: D \to D, \\
&\psi(t) \mapsto \cU_2(t)\psi_0+i \cA_2P( (\d^\alpha_x\psi)_{|\alpha| \leq 2}, (\d^\alpha_x\bar{\psi})_{|\alpha| \leq 2} )
\end{align*}
is a contraction mapping. \\

First, we have that by Proposition \ref{auxprop} 
\begin{align*}
\|\cU_2(t)\psi_0\|_D &\sleq c^{ \frac{d}{2} + \frac{4}{m(m+2)} } \|\psi_0\|_{M^{s+3+1/m}_{2,1}}.
\end{align*}
Now, for the estimate of $\rho_1^{(i)}(\cA_2 \d^\alpha_{x_j}F)$ ($i,j=1,\ldots,d$) 
it suffices to estimate $\rho_1^{(1)}(\cA_2 \d^\alpha_{x_1}F)$: indeed, by 
\eqref{estparder}
\begin{align*}
\rho_1^{(1)}(\cA_2 \d^\alpha_{x_2}F) &\sleq \rho_1^{(1)}(\cA_2 \d^\alpha_{x_1}F).
\end{align*}
Using frequency-uniform decomposition, we write
\begin{align*}
\square_k(\Psi_1 \cdots \Psi_R) &= \sum_{\B^{(R)}_{1,1}} \square_k( \square_{k^{(1)}} \, \Psi_1 \cdots \square_{k^{(R)}} \, \Psi_R) + \sum_{\B^{(R)}_{1,2}} \square_k( \square_{k^{(1)}} \, \Psi_1 \cdots \square_{k^{(R)}} \, \Psi_R).
\end{align*}
By exploiting \eqref{smoothestfin1def} and \eqref{estnlin1} for the first sum 
and \eqref{auxstr2def} and \eqref{estnlin3} for the second sum we obtain
\begin{align*}
\rho_1^{(1)}(\cA_2 \, \d_{x_1}^{\alpha}(\Psi_1\cdots\Psi_R)) &\sleq \left\| \sum_{ \B^{(R)}_{1,1} } \square_{k^{(1)}} \, \Psi_1 \cdots \square_{k^{(R)}} \, \Psi_R \right\|_{l^{1,s-r+1/2+1/m}_{\square,1,c}(L^{1,2}_{x_1,(x_j)_{j \neq 1},t})} \\ 
&+ c^{1+\frac{4}{R^2-1}} \, \left\| \sum_{ \B^{(R)}_{1,1} } \square_{k^{(1)}} \, \Psi_1 \cdots \square_{k^{(R)}} \, \Psi_R \right\|_{ l^1_{\square,1,c}( L^{\frac{R+1}{R}}_{t,x} ) } \\
&\sleq c^{1+\frac{4}{R^2-1}+d} \, R^d \|\psi\|_D^R.
\end{align*}

Next, we estimate $\rho_2^{(1)}(\cA_2(\Psi_1 \cdots \Psi_R))$ and 
$\rho_3^{(1)}(\cA_2(\Psi_1 \cdots \Psi_R))$. 
By \eqref{auxstr3def} and \eqref{strest2def} we have 
\begin{align*}
\sum_{j=2}^3 \rho_j^{(1)}(\cA_2(\Psi_1 \cdots \Psi_R)) &\sleq c^{ \frac{d}{2} + \frac{2}{m+2} + \frac{8}{m(m+2)} } \|\Psi_1 \cdots \Psi_R\|_{l^{1,s+1/m}(L^{\frac{2+m}{1+m}}_{t,x})} \\
&\stackrel{\eqref{estnlin3}}{\sleq} c^{ \frac{d}{2} + \frac{2}{m+2} + \frac{8}{m(m+2)} } R^d \|\psi\|_D^R.
\end{align*}

Then we consider $\rho_2^{(1)}(\cA_2 \, \d^2_{x_1} (\Psi_1 \cdots \Psi_R))$: 
we have
\begin{align*}
\rho_2^{(1)}(\cA_2 \, \d^2_{x_1} (\Psi_1 \cdots \Psi_R)) &\sleq \left( \sum_{\substack{k \in \Z^d \\ |k|_\infty \sgeq c}} + \sum_{\substack{k \in \Z^d \\ |k|_\infty \sleq c}} \right) \| \square_k \, \cA_2 \, \d^2_{x_1} (\Psi_1 \cdots \Psi_R) \|_{L^{m,\infty}_{x_1;(x_j)_{j \neq 1},t}} \\
&=: III + IV.
\end{align*}
Again by \eqref{strest2def} and \eqref{estnlin3} we obtain 
\begin{align*}
IV &\sleq c^{ \frac{d}{2} + \frac{2}{m+2} + \frac{4}{m(m+2)} } \|\Psi_1 \cdots \Psi_R\|_{l^{1,s+1/m}(L^{\frac{2+m}{1+m}}_{t,x})} \\
&\sleq c^{ \frac{d}{2} + \frac{2}{m+2} + \frac{4}{m(m+2)} } R^d \|\psi\|_D^R.
\end{align*}
Furthermore, we have that
\begin{align*}
III &\sleq \left( \sum_{k \in \Z^d_1} + \cdots + \sum_{k \in \Z^d_d} \right) \| \square_k \, \cA_2 \, \d^2_{x_1} (\Psi_1 \cdots \Psi_R) \|_{L^{m,\infty}_{x_1; (x_j)_{j \neq 1},t}} \\
&=: G_1(\psi) + \cdots G_d(\psi).
\end{align*}
Using the frequency-uniform decomposition, \eqref{maxfest4hdef}, 
\eqref{estnlin1} and \eqref{estnlin2} we have that
\begin{align*}
G_i(\psi) &\sleq c^{ 1+ 3 \frac{d}{2} } R^d \|\psi\|_D^R, \; i=1,\ldots,d,
\end{align*}
therefore
\begin{align*}
III &\sleq c^{ 1+ 3 \frac{d}{2} } R^d \|\psi\|_D^R,
\end{align*}

Finally, we estimate $\rho_3^{(1)}( \cA_2 \, \d^2_{x_i}(\Psi_1 \cdots \Psi_R) )$. 
It suffices to consider the case $i=1$: by \eqref{retstrestfrdec} and \eqref{danest} we have 
\begin{align*}
\| \square_k \, \cA_2 \d^2_{x_1}f\|_{L^\infty_t L^2 \cap L^{2+m}_{t,x}} & \sleq c^{ \frac{4}{m(m+2)} } \la k_1 \ra^2 \| \square_k \, f \|_{L^{\frac{2+m}{1+m}}_{t,x}},
\end{align*}
and by \eqref{auxstr1def} and \eqref{strest2} we obtain
\begin{align*}
\rho_3^{(1)}( \cA_2 \, \d^2_{x_i}(\Psi_1 \cdots \Psi_R)) &\sleq c^{ 1 + \frac{8}{m(m+2)} + \frac{d}{2} } R^d \|\psi\|_D^R.
\end{align*}

Collecting all estimates, we have
\begin{align}
\| \cF(\psi) \|_D &\sleq c^{ \frac{d}{2} + \frac{4}{m(m+2)} } \|\psi_0\|_{M^{s+3+1/m}_{2,1}} + c^{ 1+ \frac{3d}{2} +\frac{2}{m+2} + \frac{8}{m(m+2)} } \sum_{m+1 \leq R < M} c^{\frac{4}{R^2-1}} R^d \|\psi\|_D^R.
\end{align}
and for $c \geq 1$ sufficiently large we can conclude 
by a standard contraction mapping argument (see for example the proof of 
Theorem 1.1 in \cite{cazenave1990cauchy}), by choosing 
\begin{align}
\delta &> \delta_0(d,m,2) := \max\left( \frac{d}{2}+\frac{4}{m(m+2)} , \frac{1}{m} + \frac{3d}{2m} + \frac{2}{m(m+2)} + \frac{8}{m^2(m+2)} + \frac{4}{m^3} \right). \label{delta02}
\end{align}

\end{proof}

\begin{remark}
By arguing in the same way for the general case $r>2$ we end up with the condition 
\begin{align}
&\delta > \delta_0(d,m,r) 
%&:= \max\left( d\left(1-\frac{1}{r}\right)+\frac{4r}{m^2( m+2(r-1) )} , \frac{r-1}{m} + \frac{3d}{2m} + \frac{2r}{m( m+2(r-1) )} + \frac{8(r-1)^2}{m^2( m+2(r-1) )} + \frac{4(r-1)^2}{m^3} \right). \label{delta0r}
:= \max\left( d\left(1-\frac{1}{r}\right)+\frac{4r}{m^2( m+2(r-1) )} , \frac{r-1}{m} + \frac{3d}{2m} + \frac{2rm + 8(r-1)^2}{m^2( m+2(r-1) )} + \frac{4(r-1)^2}{m^3} \right). \label{delta0r}
\end{align}
\end{remark}

\begin{remark}
The quantity $\delta_0(d,l,r)$ defined in Corollary \ref{LWPBNFr} is actually 
the right-hand side of \eqref{delta0r} with $m$ replaced by $2(l-1)$.
\end{remark}

In order to prove the second part of Proposition \ref{LWPhighordschr} we will exploit another contraction mapping argument, like in the proof of Theorem 1 in \cite{hao2007well} (which in turn is based on the proof of Theorem 4.1 of \cite{kenig1993small}). In the following, we denote by a $(Q_\alpha)_{\alpha \in \Z^d}$ a fixed family of nonoverlapping cubes of size $R$ such that $\R^d = \bigcup_\alpha Q_\alpha$.

\begin{lemma}
Let $d \geq 2$ and $r \geq 2$, then the following estimates hold.
\begin{itemize}
\item \emph{(Local smoothing, homogeneous case)}
\begin{align}
\sup_{\alpha \in \Z^d} \left( \int_{Q_\alpha} \int_\R |D_x^{r-1/2} \cU_r(t)\psi_0(x)|^2 \di t \di x \right)^{1/2} &\sleq c^{r-1} R^{1/2} \|\psi_0\|_{L^2}, \label{homlocsmooth1} \\
\left\|D_x^r \int_I \cU_r(t-\tau)\psi(\tau,\cdot)\di\tau \right\|_2 &\sleq  c^{r-1} R^{1/2} \sum_{\alpha \in \Z^d} \left( \int_{Q_\alpha} \int_I |\psi(t,x)|^2 \di t \di x\right)^{1/2}; \label{homlocsmooth2}
\end{align}
\item \emph{(Local smoothing, inhomogeneous case)} the solution of the inhomogeneous Cauchy problem
\begin{align*}
-i\psi_t &= A_{c,r}\psi + f(t,x), \; t \in I, x \in \R^d,
\end{align*}
such that $\psi_0 \equiv 0$ satisfies
\begin{align} \label{inhomlocsmooth}
\sup_{\alpha \in \Z^d} \|D_x^{2(r-1)} \psi\|_{L^2_x(Q_\alpha); L^2_t(I)} &\sleq c^{2(r-1)} \, R T^{1/(4d)} \sum_{\alpha \in \Z^d} \|f\|_{L^2_x(Q_\alpha); L^2_t(I)}
\end{align}
\item \emph{(Maximal function estimate)} For any $s>d+\frac{1}{2}$ we have
\begin{align} \label{maxfun}
\left( \int_{\R^d} \sup_{|t| \sleq c^{2(r-1)}} |\cU_r(t)\psi_0(x)|^2 \di x \right)^{1/2} &\sleq c^{d\left(1-\frac{1}{r}\right)} \|\psi_0\|_{H^s}.
\end{align}
\end{itemize}
\end{lemma}

\begin{proof}[Proof (sketch)]
The proof in the case $r=2$ can be obtained simply by rescaling Lemma 3, Lemma 4, Lemma 5 and Lemma 6 of \cite{hao2007well}. The proof in the case $r>2$ can be obtained by considering the operator $\cU_r(t)$ and $\cA_r(t)$ instead of $\cU_2(t)$ and $\cA_2(t)$.
\end{proof}

\begin{proof}[Proof (Proposition \ref{LWPhighordschr}, part $(ii)$, case $r=2$)] 
We will prove the result only for $s=s_0$, since the general case follows from commutator estimates. For simplicity, we only deal with the case 
\begin{align*}
P( (\d^\alpha_x\psi)_{|\alpha| \leq 2}, (\d^\alpha_x\bar{\psi})_{|\alpha| \leq 2} ) &= \d_{x_j}^2\psi \, \d_{x_k}^2\psi \, \d_{x_m}^2\psi.
\end{align*}
More precisely, we fix a positive constant $\nu<1/3$, and we define the space $Z_I^\delta$ of all function $\phi:I \times \R^d \to \C$ such that the following three conditions hold
\begin{align} 
\|\phi\|_{L^\infty(I) H^{s_0}} &\leq c^{-\delta}, \label{ZIdelta1} \\
\sum_{|\beta|=s_0+1/2} \sup_{\alpha \in \Z^d} \left( \int_I \int_{Q_\alpha} |\d^\beta_x\phi(t,x)|^2 \di x \di t \right)^{1/2} &\leq T^\nu, \label{ZIdelta2} \\
\left( \sup_{t \in I} \sup_{x \in Q_\alpha} |D^2_x\phi(t,x)|^2 \right)^{1/2} &\leq c^{-\delta}. \label{ZIdelta3}
\end{align}
We want to show that the map 
\begin{align*}
\cF &: Z_I^\delta \to Z_I^\delta, \\
&\psi(t) \mapsto \cU_2(t)\psi_0+i \cA_2P( (\d^\alpha_x\psi)_{|\alpha| \leq 2}, (\d^\alpha_x\bar{\psi})_{|\alpha| \leq 2} )
\end{align*}
is a contraction mapping. \\

We can observe that for any $\beta \in \Z^d$ with $|\beta|=s_0-\frac{3}{2}$
\begin{align*}
\d_x^\beta(\d_{x_j}^2\psi \, \d_{x_k}^2\psi \, \d_{x_m}^2\psi) &= \d_x^\beta\d_{x_j}^2\psi \, \d_{x_k}^2\psi \, \d_{x_m}^2\psi + \d_{x_j}^2\psi \, \d_x^\beta\d_{x_k}^2\psi \, \d_{x_m}^2\psi + \d_{x_j}^2\psi \, \d_{x_k}^2\psi \, \d_x^\beta\d_{x_m}^2\psi \nonumber \\
&\; \; \; + R((\d_x^\gamma\psi)_{2\leq|\gamma|\leq s_0-1/2}).
\end{align*}
Now, for any $\psi \in Z_I^\delta$ we have
\begin{align}
&\sum_{|\beta|=s_0+1/2} \sup_{\alpha \in \Z^d} \left( \int_I \int_{Q_\alpha} |\d^\beta_x\psi(t,x)|^2 \di x \di t \right)^{1/2} \nonumber \\
&\sleq \sum_{|\beta|=s_0+1/2} \sup_{\alpha \in \Z^d} \left( \int_I \int_{Q_\alpha} |\cU_2(t)\d^\beta_x\psi_0(x)|^2 \di x \di t \right)^{1/2} \nonumber \\
&\; \; \; + \sum_{|\beta|=s_0+1/2} \sup_{\alpha \in \Z^d} \left( \int_I \int_{Q_\alpha} \left| \int_0^t \cU_2(t-\tau) \d_x^\beta(\d_{x_j}^2\psi \, \d_{x_k}^2\psi \, \d_{x_m}^2\psi) \di\tau \right|^2 \di x \di t \right)^{1/2} \nonumber \\
&\stackrel{\eqref{homlocsmooth1},\eqref{homlocsmooth2}}{\sleq} c T^{1/3} \|\psi_0\|_{H^{s_0}} + c^2 T^{1/(4d)} \sum_{|\beta_0|=s_0-3/2} \sum_{j,k,m=1}^d \sum_{\alpha \in \Z^d} \| \d_x^\beta\d_{x_j}^2\psi \, \d_{x_k}^2\psi \, \d_{x_m}^2\psi \|_{L^2_x(Q_\alpha; L^2_t(I))} \nonumber \\
&\; \; \; + c^2 \int_0^T \|D_x^{1/2} R((\d_x^\gamma\psi)_{2\leq|\gamma|\leq s_0-1/2}) \|_{L^2} \di t \nonumber \\
&\sleq c T^{1/3} \|\psi_0\|_{H^{s_0}} + c^2 T^{1/(4d)} \sum_{|\beta_0|=s_0+1/2} \sup_{\alpha \in \Z^d} \left( \int_I \int_{Q_\alpha} |\d_x^\beta\psi|^2 \di x \di t \right)^{1/2} \left( \sum_{\alpha \in \Z^d} \sup_{t \in I} \sup_{x \in Q_\alpha} |D_x^2\psi|^2 \right) \nonumber \\
&\; \; \; + c^2 T \sup_{t \in I} \|\psi\|_{H^{s_0}}^3 \nonumber \\
&\sleq c^{1-\delta} T^{1/3} + c^2 T^{1/(4d)} T^\nu c^{-2\delta} +  c^2 T c^{-3\delta} \nonumber \\
&\leq T^\nu, \label{ZIdeltaterm2}
\end{align}
where in the last inequality we have chosen $\delta \gg 1$ such that
\begin{align} \label{ZIdeltaterm2cond}
c^{1-\delta} T^{-\nu+1/3} + c^{2(1-\delta)} T^{1/(4d)} +  c^{2-3\delta} T^{1-\nu} &\sleq 1, \; \; T =\cO(c^{2(r-1)}). 
\end{align}

Next, we have that for any $\psi \in Z_I^\delta$
\begin{align}
\|\psi\|_{L^\infty(I)H^{s_0}} &\leq \|\psi_0\|_{H^{s_0}} + \sup_{t \in I} \int_0^t \|\cU_2(t-\tau) d_{x_j}^2\psi(\tau) \, \d_{x_k}^2\psi(\tau) \, \d_{x_m}^2\psi(\tau) \|_{L^2} \di \tau \nonumber \\
&\; \; \; + \sup_{t \in I} \left\| D_x^{3/2} \int_0^t \cU_2(t-\tau) D_x^{s_0-3/2}d_{x_j}^2\psi(\tau) \, \d_{x_k}^2\psi(\tau) \, \d_{x_m}^2\psi(\tau) \di \tau  \right\|_{L^2} \nonumber \\
&\stackrel{\eqref{homlocsmooth2}}{\sleq} \|\psi_0\|_{H^{s_0}} + T \sup_{t \in I} \| \d_{x_j}^2\psi(t) \, \d_{x_k}^2\psi(t) \, \d_{x_m}^2\psi(t) \|_{L^2} \nonumber \\
&\; \; \; + c^{r-1} \sum_{\alpha \in \Z^d} \left( \int_{Q_\alpha} \int_I |D_x^{s_0-3/2}( \d_{x_j}^2\psi(t) \, \d_{x_k}^2\psi(t) \, \d_{x_m}^2\psi(t) )|^2\di t \di x \right)^{1/2} \nonumber \\
&\sleq \|\psi_0\|_{H^{s_0}} + T \sup_{t \in I} \|\psi\|^3_{H^{\frac{d}{3}+2}} \nonumber \\
&\; \; \; + c \sum_{j,k,m=1}^d \sum_{\alpha \in \Z^d} \left( \int_{Q_\alpha} \int_I |D_x^{s_0-3/2}\d_{x_j}^2\psi(t) \, \d_{x_k}^2\psi(t) \, \d_{x_m}^2\psi(t) |^2\di t \di x \right)^{1/2} \nonumber \\
&\; \; \; + c \sum_{\alpha \in \Z^d} \left( \int_{Q_\alpha} \int_I |R(D_x^\gamma\psi)_{2\leq|\gamma|\leq s_0-1/2}|^2 \di t \di x \right)^{1/2} \nonumber \\
&\sleq \|\psi_0\|_{H^{s_0}} + T \|\psi\|^3_{L^\infty(I)H^{\frac{d}{3}+2}} \nonumber \\
&\; \; \;+ c \sum_{|\beta|=s_0+1/2} \sup_{\alpha \in \Z^d} \left( \int_I \int_{Q_\alpha} |\d_x^\beta\psi|^2 \di x \di t \right)^{1/2} \; \sum_{\alpha \in \Z^d} \sup_{t \in I} \sup_{x \in Q_{\alpha}} |D_x^2\psi|^2 \nonumber \\
&\; \; \; + c T^{1/2} \|\psi\|^3_{L^\infty(I) H^{s_0}} \nonumber \\
&\sleq \|\psi_0\|_{H^{s_0}}  +  (T+cT^{1/2}) c^{-3\delta} + cT^\nu c^{-2\delta} \nonumber \\
&\sleq c^{-\delta}, \label{ZIdeltaterm1}
\end{align}
where in the last inequality we have chosen $\delta \gg 1$ such that
\begin{align} \label{ZIdeltaterm1cond}
(T+cT^{1/2}) c^{-3\delta} + T^\nu c^{1-2\delta}  &\sleq \frac{1}{2}, \; \; T =\cO(c^{2(r-1)}). 
\end{align}

Then, we have that for any $\psi \in Z_I^\delta$
\begin{align}
&\left( \sum_{\alpha \in \Z^d} \sup_{t \in I} \sup_{x \in Q_\alpha} |D^2_x\psi(t,x)|^2 \right)^{1/2} \nonumber \\
&\stackrel{T \sleq c^{2(r-1)},\eqref{maxfun} }{\sleq} \|\psi_0\|_{H^{s_0}} + c^{d\left(1-\frac{1}{r}\right)} T \|\psi\|^3_{L^\infty(I)H^{s_0}} \nonumber \\
&\sleq \|\psi_0\|_{H^{s_0}} + c^{d\left(1-\frac{1}{r}\right)} T c^{-3\delta} \nonumber \\
&\sleq c^{-\delta}, \label{ZIdeltaterm3}
\end{align}
where in the last inequality we have chosen $\delta \gg 1$ such that
\begin{align} \label{ZIdeltaterm3cond}
c^{d\left(1-\frac{1}{r}\right)-2\delta} T   &\sleq 1, \; \; T =\cO(c^{2(r-1)}).
\end{align}

Finally, if for any $\phi \in Z_I^\delta$ we set $\Lambda_T(\phi)$ as the maximum between the three following quantities,
\begin{equation*}
\sum_{\alpha \in \Z^d} \sup_{t \in I} \sup_{x \in Q_\alpha} |D^2_x\psi(t,x)|^2,
\end{equation*}
\begin{equation*}
\|\phi\|_{L^\infty(I)H^{s_0}}, 
\end{equation*}
\begin{equation*}
c^{-\delta} T^{-\nu} \sum_{|\beta|=s_0+1/2} \sup_{\alpha \in \Z^d} \left( \int_I \int_{Q_\alpha} |\d^\beta_x\phi(t,x)|^2 \di x \di t \right)^{1/2},
\end{equation*}
we can observe that for any $\phi_1,\phi_2 \in Z_I^\delta$
\begin{align*}
\Lambda_T( \cF(\phi_1)-\cF(\phi_2) ) &\leq K T^\nu c^{-2\delta} \Lambda_T(\phi_1-\phi_2),
\end{align*}
where $K$ is a positive constant which does not depend on $c$. Hence, if we choose $\delta \gg 1$ such that \eqref{ZIdeltaterm1cond}, \eqref{ZIdeltaterm2cond}, \eqref{ZIdeltaterm3cond} and 
\begin{align} \label{ZIdeltacontrcond}
K T^\nu c^{-2\delta} &\leq \frac{1}{2}
\end{align}
hold true, we can conclude.
\end{proof}

\section{Long time approximation of radiation solutions} \label{longtappr}

Now we want to exploit the result of the previous section 
in order to deduce some consequences about the dynamics of the NLKG 
equation \eqref{dsa} on $M=\R^d$, $d \geq 2$, in the nonrelativistic limit. \\

Consider the \emph{simplified system}, that is the Hamiltonian $H_r$ 
in the notations of Theorem \ref{normformgavthm}, where we neglect the 
remainder:
\begin{align*}
H_{simp} &:= h_0+\epsilon(h_1+ \la F_1 \ra)+ \sum_{j=2}^{r} \epsilon^j(h_j+Z_j).
\end{align*}
We recall that in the case of the NLKG the simplified system is 
actually the NLS (given by $h_0+\epsilon(h_1+ \la F_1 \ra)$), 
plus higher-order normalized corrections. Now let $\psi_r$ be a solution of  
\begin{align} \label{simpleq}
-i \,\dot \psi_r \, &= \, X_{H_{simp}}(\psi_r),
\end{align}
then $\psi_a(t,x):=\cT^{(r)}(\psi_r(c^2t,x))$ solves
\begin{align} \label{appreq}
\dot \psi_a&= i c\jap{\grad}_c \psi_a + 
\frac{\lambda}{2l} \left( \frac{c}{\jap{\grad}_c} \right)^{1/2} \, 
\left[ \left( \frac{c}{\jap{\grad}_c} \right)^{1/2} \frac{\psi_a+\bar\psi_a}{\sqrt{2}} \right]^{2l-1}
- \frac{1}{c^{2r}} X_{ \cT^{(r)*}\cR^{(r)} }(\psi_a,\bar\psi_a),
\end{align}
that is, the NLKG plus a remainder of order $c^{-2r}$ (in the following 
we will refer to equation \eqref{appreq} as \emph{approximate equation}, 
and to $\psi_a$ as the \emph{approximate solution} of the original NLKG). 
We point out that the original NLKG and the approximate equation differ 
only by a remainder of order $c^{-2r}$, which is evaluated on the approximate 
solution. This fact is extremely important: indeed, if one can prove the 
smoothness of the approximate solution (which often is easier to check 
than the smoothness of the solution of the original equation), 
then the contribution of the remainder may be considered small 
in the nonrelativistic limit. This property is rather general, 
and has been already applied in the framework of normal form theory 
(see for example \cite{bambusi2002nonlinear}). \\
\indent Now let $\psi$ be a solution of the NLKG equation \eqref{dsa} 
with initial datum $\psi_0$, and let $\delta:=\psi-\psi_a$ be the error 
between the solution of the approximate equation and the original one. 
One can check that $\delta$ fulfills 
\begin{align*}
\dot \delta &= i c \jap{\grad}_c \delta +
[ P(\psi_a+\delta,\bar\psi_a+\bar\delta)-P(\psi_a,\bar\psi_a) ]+
\frac{1}{c^{2r}} X_{ \cT^{(r)*}\cR^{(r)} }(\psi_a(t),\bar\psi_a(t) ),
\end{align*}
where 
\begin{align} \label{nonlin}
P(\psi,\bar\psi)&= \frac{\lambda}{2l} \left( \frac{c}{\jap{\grad}_c} \right)^{1/2}
\left[ \left( \frac{c}{\jap{\grad}_c} \right)^{1/2} \frac{\psi+\bar\psi}{\sqrt{2}} \right]^{2l-1}. 
\end{align}
Thus we get
\begin{align}
\dot \delta &= i \,c \jap{\grad}_c\delta + dP(\psi_a(t))\delta + \cO(\delta^2) + \cO\left(\frac{1}{ c^{2r} }\right); \nonumber \\
\delta(t)&= e^{itc\jap{\grad}_c}\delta_0 + 
\int_0^{t}e^{i(t-s)c\jap{\grad}_c}dP(\psi_a(s))\delta(s)\di s +
\cO(\delta^2)+\cO\left(\frac{1}{c^{2r}}\right). \label{erreqrem}
\end{align}
By applying Gronwall inequality to \eqref{erreqrem} we can obtain an approximation result which is valid only locally uniformly in time, namely up to times of 
order $\cO(1)$ (see Theorem 2.3 of \cite{pasquali2017dynamics}). \\

Observe that the evolution of the error $\delta$ between the approximate 
solution $\psi_a$, namely the solution of \eqref{appreq}, and the original 
solution $\psi$ of \eqref{dsa} is described by
\begin{align} \label{erreq}
\dot\delta(t)&= i \, c\jap{\grad}_c\delta(t) + dP(\psi_a(t))\delta(t); \\
\delta(t) &= e^{itc\jap{\grad}_c}\delta_0+\int_0^{t}e^{i(t-s)c\jap{\grad}_c}dP(\psi_a(s))\delta(s)\di s,
\end{align}
up to a remainder which is small, if we assume the smoothness of $\psi_a$. \\

Now we study the evolution of the error for long (that means, $c$-dependent) 
time intervals. 

We pursue such a program by a perturbative argument, considering a small 
radiation solution $\psi_r=\eta_{rad,r}$ of the normalized system 
\eqref{simpleq} that exists up to times of order $\cO(c^{2(r-1)})$, $r>1$. \\

As an application of Proposition \ref{strlin}, we consider the following 
case. Fix $r>1$, let $\sigma>0$ and let $\psi_r = \eta_{rad}$ be a radiation 
solution of \eqref{simpleq}, namely such that 
\begin{align} \label{hyprad}
\eta_{rad,0}:=\eta_{rad}(0) &\in H^{k+k_0+\sigma+d/2}(\R^d),
\end{align}
where $k_0>0$ and $k \gg 1$ are the ones in Theorem \ref{normformgavthm}.

Let $\delta(t)$ be a solution of \eqref{erreq}; then by Duhamel formula
\begin{align}
\delta(t) &:= \cU(t,0) \delta_0 = e^{it c\jap{\grad}_c}\delta_0 
+\int_0^t e^{i(t-s) c\jap{\grad}_c} \di P(\psi_a(s)) \cU(s,0) \delta_0 \di s. \label{duhamerr}
\end{align}
Now fix $T\sleq c^{2(r-1)}$; we want to estimate the local-in-time norm 
in the space $L^\infty([0,T])H^{k}(\R^d)$ of the error $\delta(t)$. 

By \eqref{strestkg} we can estimate the first term. We can estimate the 
second term by \eqref{retstrkg}: hence for any $(p,q)$ 
Schr\"odinger-admissible exponents

\begin{align*}
& \left\| \int_0^{t}e^{i(t-s) c\jap{\grad}_c}dP(\psi_a(s)) \delta(s) \di s \right\|_{L^\infty_t([0,T]) H^{k}_x} \\
&\sleq c^{\frac{1}{q}-\frac{1}{p}-\frac{1}{2}} \| \jap{\grad}_c^{\frac{1}{p}-\frac{1}{q}+\frac{1}{2}} \di P(\psi_a(t)) \delta(t) \|_{L^{p'}_t([0,T]) W^{k,q'}_x} \\
&\sleq c^{\frac{1}{q}-\frac{1}{p}-\frac{1}{2}} \| \jap{\grad}_c^{\frac{1}{p}-\frac{1}{q}+\frac{1}{2}}  \di P(\eta_{rad}(c^2t)) \delta(t) \|_{L^{p'}_t([0,T]) W^{k,q'}_x}  \\
&\; \; \; + c^{\frac{1}{q}-\frac{1}{p}-\frac{1}{2}} \| \jap{\grad}_c^{\frac{1}{p}-\frac{1}{q}+\frac{1}{2}}  [\di P(\psi_a(t))-\di P(\eta_{rad}(c^2t)) ] \delta(t) \|_{L^{p'}_t([0,T]) W^{k,q'}_x} \\
&=: I_p + II_p,
\end{align*}

but recalling \eqref{nonlin} one has that

\begin{align*}
I_p &\sleq \frac{|\lambda|}{2^{l-1/2}(2l)(2l-1)} c^{\frac{1}{q}-\frac{1}{p}+\frac{1}{2}}
\left\| \jap{\grad}_c^{\frac{1}{p}-\frac{1}{q}-\frac{1}{2}} 
\left[ \left( \frac{c}{\jap{\grad}_c} \right)^{1/2} (\eta_{rad}+\bar\eta_{rad}) \right]^{2(l-1)}
\, \delta(t) \right\|_{L^{p'}_t([0,T]) W^{k,q'}_x}.
\end{align*}

Now fix a real number $\rho \in ]0,1[$, and choose
\begin{align}
p &= 2+\rho, \label{pexp} \\
q=\frac{2dp}{dp-4} &= \frac{4d+2d\rho}{2d+d\rho-4} = 2+\frac{8}{2d-4+d\rho}, \label{qexp}
\end{align}
we get (since $\|(c/\jap{\grad}_c)^{\frac{1}{q}-\frac{1}{p}-\frac{1}{2}}\|_{L^{q'} \to L^{q'}} \leq 1$)
\begin{align*}
I_{2} &\leq \frac{|\lambda|}{2^{l-1/2}(2l)(2l-1)} 
\left\| \left[ \left( \frac{c}{\jap{\grad}_c} \right)^{1/2} (\eta_{rad}(c^2 t)+\bar\eta_{rad}(c^2 t)) \right]^{2(l-1)}
\, \delta(t) \right\|_{L^{ 2-\frac{\rho}{1+\rho} }_t([0,T]) W^{k,q'}_x}.
\end{align*}

Now, since by H\"older inequality
\begin{align*}
&\left\| \left[ \left( \frac{c}{\jap{\grad}_c} \right)^{1/2} (\eta_{rad}(c^2 t)+\bar\eta_{rad}(c^2 t)) \right]^{2(l-1)}
\, \delta(t) \right\|_{L^{ 2-\frac{\rho}{1+\rho}  }_t([0,T]) W^{k,q'}_x} \\
&\leq \left\| \left[ \left( \frac{c}{\jap{\grad}_c} \right)^{1/2} (\eta_{rad}(c^2 t)+\bar\eta_{rad}(c^2 t)) \right]^{2(l-1)} \right\|_{L_t^{ 2-\frac{\rho}{1+\rho} }([0,T]) W^{k,d(1+\rho/2)}_x} 
\, \|\delta(t)\|_{L^\infty_t([0,T]) H^{k}_x},
\end{align*}
and by Sobolev product theorem 
(recall that $l \geq 2$, and that $k \gg 1$) we can deduce that
\begin{align*}
& \left\| \left[ \left( \frac{c}{\jap{\grad}_c} \right)^{1/2} (\eta_{rad}(c^2 t)+\bar\eta_{rad}(c^2 t)) \right]^{2(l-1)} \right\|_{L_t^{ 2-\frac{\rho}{1+\rho} }([0,T]) W^{k,d(1+\rho/2)}_x} \\
&\leq  \left[ \int_0^T 
\left\| \left[ \left( \frac{c}{\jap{\grad}_c} \right)^{1/2} (\eta_{rad}(c^2 t)+\bar\eta_{rad}(c^2 t)) \right] \right\|_{W^{k,d(1+\rho/2)}_x }^{2(l-1)\left( 2-\frac{\rho}{1+\rho} \right)} \di t \right]^{\frac{1}{\left( 2-\frac{\rho}{1+\rho} \right)}} \\
&\leq \left\|  \eta_{rad}(c^2 t)+\bar\eta_{rad}(c^2 t) \right\|^{2(l-1)}_{L^{2(l-1)\left( 2-\frac{\rho}{1+\rho} \right)}_t([0,T]) W^{k,d(1+\rho/2)}_x} ,
\end{align*}
but since by Proposition \ref{modspaceprop} we have that for any $\sigma>0$
\begin{align*}
L^{2(l-1)\left( 2-\frac{\rho}{1+\rho} \right)}_t([0,T]) W^{k,d(1+\rho/2)}_x &\supseteq L^{2(l-1)\left( 2-\frac{\rho}{1+\rho} \right)}_t([0,T]) M^k_{d(1+\rho/2),1,x} \\
&\supseteq L^{2(l-1)\left( 2-\frac{\rho}{1+\rho} \right)}_t([0,T]) M^k_{2,1,x} \\
&\supseteq L^{2(l-1)\left( 2-\frac{\rho}{1+\rho} \right)}_t([0,T]) H^{k+\sigma+d/2}_x \\
&\supseteq L^{\infty}_t([0,T]) H^{k+\sigma+d/2}_x,
\end{align*}
we have that 
\begin{align} \label{estrad}
\| \eta_{rad} \|^{2(l-1)}_{L^{2(l-1)\left( 2-\frac{\rho}{1+\rho} \right)}_t([0,T]) W^{k,d(1+\rho/2)}_x} &\sleq T^{\frac{1+\rho}{2+\rho}} \| \eta_{rad} \|^{2(l-1)}_{ L^\infty_t([0,T])H^{k+\sigma+d/2}_x },
\end{align}
but by Corollary \ref{LWPBNFr} the right-hand side of \eqref{estrad} 
is finite and does not depend on $c \geq 1$ for 
\begin{align}
\|\eta_{rad,0} \|_{H^{k+k_0+\sigma+d/2}_x} &\sleq c^{-\alpha}, \label{extrahyp1} \\
\alpha &> \max\left( \delta_0(d,l,r), \delta_1(d,l,r) , \frac{r-1}{l-1}  \right) := \alpha^\ast(d,l,r) . \label{extrahyp2}
%\alpha>\frac{r-1}{(l-1)} \frac{1+\rho}{2+\rho}
\end{align}
where $c \geq c_0$ is sufficiently large, and where 
$\delta_0(d,l,r)$  and $\delta_1(d,l,r)$ are defined in Corollary \ref{LWPBNFr}. \\

Furthermore, via \eqref{CTthm} one can show that there exists
$c_{r,k}>0$ sufficiently large such that for $c \geq c_{r,k}$ 
the term $II_{2}$ can be bounded by $\frac{1}{c^2} \; I_2$. \\

This means that we can estimate the $L^\infty([0,T])H^{k}$ norm of the error 
only for a small (with respect to $c$) radiation solution, 
which is the statement of Proposition \ref{NLKGtoNLSrad}.

To summarize, we get the following result.

\begin{proposition} \label{NLKGtoNLSrad}
Consider \eqref{NLKGhamnew} on $\R^d$, $d \geq 2$. 
Let $r>1$, and fix $k_1 \gg 1$. 
Assume that $l \geq 2$ and $r < \frac{d}{2}(l-1)$. 
Then $\exists$ $k_0=k_0(r)>0$ such that for any $k \geq k_1$ 
and for any $\sigma>0$ the following holds: 
consider the solution $\eta_{rad}$ of \eqref{simpleq} 
with initial datum $\eta_{rad,0} \in H^{k+k_0+\sigma+d/2}(\R^d)$, and 
call $\delta$ the difference between the solution of 
the approximate equation \eqref{appreq} and the original solution of 
the Hamilton equation for \eqref{NLKGhamnew}. 
Assume that $\delta_0:=\delta(0)$ satisfies
\[ \|\delta_0\|_{H^k_x} \sleq \frac{1}{c^2}. \]
Then there exist $\alpha^\ast:=\alpha^\ast(d,l,r)>0$ 
and there exists  $c^\ast:=c^\ast(r,k) > 1$, 
such that for any $\alpha> \alpha^\ast$ and for any $c > c^\ast$, 
if $\eta_{rad,0}$ satisfies
\begin{align*}
\|\eta_{rad,0}\|_{H^{k+k_0+\sigma+d/2}} &\sleq c^{-\alpha},
\end{align*}
then
\begin{align*}
\sup_{t\in [0,T]} \|\delta(t)\|_{H^k_x} &\sleq \frac{1}{c^2}, \; \; T \sleq c^{2(r-1)}.
\end{align*}
\end{proposition}

By exploiting \eqref{CTthm} and Proposition \ref{NLKGtoNLSrad}, 
we obtain Theorem \ref{NLKGtoNLSradthm}.

\bibliography{P_NLKG_2017_part2}
\bibliographystyle{alpha}

\end{document}